\documentclass[article]{siamart}

\usepackage[shortlabels]{enumitem}
\usepackage{bm}
\usepackage{soul}
\usepackage{amssymb}
\usepackage{mathtools}
\usepackage{pgf,tikz}
\usepackage{tikz}
\usetikzlibrary{shapes.geometric}
\usetikzlibrary{arrows}
\usetikzlibrary{shapes.misc}
\usepackage{amsfonts}
\usepackage{float}
\usepackage{graphicx}
\usepackage{epstopdf}
\usepackage{subcaption}
\usepackage{soul}
\ifpdf
\DeclareGraphicsExtensions{.eps,.pdf,.png,.jpg}
\else
\DeclareGraphicsExtensions{.eps}
\fi

\newtheorem{remark}{Remark}
\newtheorem{example}{Example}
\tikzset{cross/.style={cross out, draw=black, minimum size=2*(#1-\pgflinewidth), inner sep=0pt, outer sep=0pt}, cross/.default={3pt}}

\usepackage{color}
\definecolor{Gabriele}{rgb}{0, 0, 1}
\definecolor{Tommaso}{rgb}{1, 0, 0}

\newcommand{\op}[0]{{\rm op}}
\usepackage{algorithm}
\usepackage{algorithmic}

\newcommand{\spa}{{\rm span}}
\newcommand{\Int}{{\rm int}}

\newcommand{\psib}{{\pmb \psi}}

\newcommand{\C}{{\mathbb{C}}}
\newcommand{\mbI}{{\mathbb{I}}}
\newcommand{\mbA}{{\mathbb{A}}}
\newcommand{\mbT}{{\mathbb{T}}}
\newcommand{\mbG}{{\mathbb{G}}}
\newcommand{\mbX}{{\mathcal{X}}}
\newcommand{\wmbX}{{\widehat{\mathcal{X}}}}
\newcommand{\wlambda}{{\widetilde{\lambda}}}
\newcommand{\wm}{\widetilde{m}}
\newcommand{\wt}{\widetilde{t}}
\newcommand{\wT}{\widetilde{T}}

\newcommand{\weps}{\widetilde{\varepsilon}}
\newcommand{\wla}{\widetilde{\lambda}}
\newcommand{\bareps}{\bar{\varepsilon}}
\newcommand{\Span}{\mathrm{span}}

\makeatletter
\newsavebox{\@brx}
\newcommand{\llangle}[1][]{\savebox{\@brx}{\(\m@th{#1\langle}\)}%
  \mathopen{\copy\@brx\kern-0.5\wd\@brx\usebox{\@brx}}}
\newcommand{\rrangle}[1][]{\savebox{\@brx}{\(\m@th{#1\rangle}\)}%
  \mathclose{\copy\@brx\kern-0.5\wd\@brx\usebox{\@brx}}}
\makeatother

\makeatletter
\newsavebox{\@brxx}
\newcommand{\llang}[1][]{\savebox{\@brxx}{(}%
  \mathopen{\copy\@brxx\kern-0.5\wd\@brxx\usebox{\@brxx}}}
\newcommand{\rrang}[1][]{\savebox{\@brxx}{)}%
  \mathclose{\copy\@brxx\kern-0.5\wd\@brxx\usebox{\@brxx}}}
\makeatother

\usepackage{array}
\usepackage{makecell}
\newcolumntype{x}[1]{>{\centering\arraybackslash}p{#1}}
\usepackage{tikz}

\newcommand{\TheTitle}{Spectral Substructured Two-level Domain Decomposition Methods} 
\newcommand{\TheAuthors}{G. Ciaramella and T. Vanzan}

\headers{Substructured 2-level DD Methods}{Ciaramella and Vanzan}

\title{{\TheTitle}}
\author{G. Ciaramella\thanks{Politecnico di Milano, Italy ({\tt gabriele.ciaramella@polimi.it}).}
	\and 
	T. Vanzan\thanks{CSQI, Ecole Polytechnique F\'{e}d\'{e}rale de Lausanne, Switzerland ({\tt tommaso.vanzan@epfl.ch}).}
}

\ifpdf
\hypersetup{
	pdftitle={\TheTitle},
	pdfauthor={\TheAuthors}
}
\fi

\begin{document}
	
\maketitle
	
\begin{abstract} 
Two-level domain decomposition (DD) methods are very powerful techniques for the 
efficient numerical solution of partial differential equations (PDEs). 
A two-level domain decomposition method requires two main components: 
a one-level preconditioner (or its corresponding smoothing iterative method), 
which is based on domain decomposition techniques, and a coarse correction step, 
which relies on a coarse space. The coarse space must properly represent the error 
components that the chosen one-level method is not capable to deal with. 
In the literature most of the works introduced efficient coarse spaces obtained 
as the span of functions defined on the entire space domain of the considered PDE. 
Therefore, the corresponding two-level preconditioners and iterative methods are 
defined in volume.

In this paper, a new class of substructured two-level methods is introduced,
for which both domain decomposition smoothers and coarse correction steps
are defined on the interfaces (or skeletons). This approach has several advantages. 
On the one hand, the required computational effort is
cheaper than the one required by classical volumetric two-level methods.
On the other hand, it allows one to use some of the well-known efficient coarse spaces proposed in the literature.
While analyzing in detail the new substructured methods, we present a new convergence analysis for two-level iterative methods, which covers the proposed substructured framework. Further, we study the asymptotic optimality of coarse spaces both theoretically and numerically using deep neural networks. 
Numerical experiments demonstrate the effectiveness of the proposed new numerical framework.
\end{abstract}
%
%
\maketitle
\section{Introduction}
Consider a linear problem of the form $A u = f$, which we assume well posed in a vector space $V$.
To define a two-level method for the solution to this problem, a one-level method and
a coarse-correction step are required.

One-level methods are generally based on a splitting technique: 
the operator $A : V \rightarrow V$ is decomposed as $A = M - N$,
where $M : V \rightarrow V$ is assumed invertible. This splitting leads to a stationary iteration,
namely $u^{k+1} = M^{-1}N u^k + M^{-1} f$, for $k=0,1,\dots$, and to a preconditioned system $M^{-1}A u = M^{-1}f$. 
These are strongly related, since the stationary iteration, if it converges, produces the solution
of the preconditioned system; see, e.g., \cite{Gabriele-Martin} and references therein.
Notice that we have tacitly used the term ``method'' with different meanings. 
On the one hand, a stationary method is a fixed-point iteration method whose goal is to obtain
the solution $u$.
On the other hand, a preconditioner is a transformation method that aims at transforming
the considered system to a new better conditioned one. 
Indeed, when talking about preconditioning, it is always implicitly assumed that the preconditioned system
is solved by a Krylov iteration. Similarly, a Krylov method can be used to accelerate a stationary iteration method.
For one-level methods (based on the same operator $M$), a precise relation makes these
two solution strategies equivalent.
Notice that one-level Domain Decomposition (DD) methods can be generally obtained by a splitting $A = M - N$, hence, they can be used as stationary iterations or preconditioners; see, e.g., \cite{Lions1,Lions2,DoleanBook,quarteroni1999domain,ToselliWidlund,DDLOGO,Gander:2008}.
Unfortunately, DD methods are in general not scalable and a coarse correction step is often desirable.
See, e.g., \cite{Chaouqui2018,CiaramellaGander,CiaramellaGander2,CiaramellaGander3,CHS1,CHS2} for exceptions and detailed scalability and non-scalability analyses.

A two-level method is characterized by the combination of a one-level method, defined on $V$,
and a coarse correction step, performed on a coarse space $V_c$.
The coarse space $V_c$ is finite dimensional and it must satisfy the
condition $\dim V_c~\ll~\dim V$. 
The mappings between $V$ and $V_c$ are realized by a restriction operator $R : V \rightarrow V_c$
and a prolongation operator $P : V_c \rightarrow V$.
In general, the restriction of $A : V \rightarrow V$ on $V_c$ is defined 
as $A_c=RAP$, which is assumed to be an invertible matrix.

Now, we distinguish two cases: a two-level stationary method and a two-level preconditioning method.
In the first case, a stationary method is used as first-level method.
After each stationary iteration, which produces an approximation $u_{app}$,
the residual $r = f - Au_{app}$ is mapped from $V$ to $V_c$, 
the coarse problem $A_c e = R r$ is solved to get $e \in V_c$,
and the coarse correction step is defined as $u_{new} = u_{app} + P e$.
This correction provides the new approximation $u_{new}$. 
By repeating these operations iteratively, one gets a two-level stationary method.
The preconditioner corresponding to this method is denoted by $M_{s,2L}$.
Notice that this idea is very closely related to two-grid methods. 
In the second case, the first-level method is purely a preconditioner $M^{-1}$.
The corresponding two-level preconditioning method, denoted by $M_{2L}$ is
obtained in an additive way: the one-level preconditioner $M^{-1}$
is added to the coarse correction matrix $PA_c^{-1} R$.
When used with appropriate implementations, the two preconditioners $M_{2L}$ and $M_{s,2L}$
require about the same computational effort per Krylov iteration. However, their different structures
can lead to different performances of Krylov methods.

The literature about two-level DD methods is very rich. See, e.g., 
\cite{Chaouqui2018Coarse,Chaouqui2019,CiaramellaRefl,Dubois2012,gander2014new,GHS2018,gander2017shem,
gander2019song,GanderVanC2019,GanderVanzan19},
for references considering DD stationary methods, and, e.g.,
\cite{Aarnes2002,Bjorstad2018,Dohrmann2008,DoleanDTN,Galvis3,Galvis,Galvis2,gander2015analysis,
Klawonn2015,Klawonn,Spillane2011,Spillane2014,Zampini2017},
for references considering DD preconditioners.
See also general classical references as \cite{DoleanBook,quarteroni1999domain,ToselliWidlund} and \cite{greenbaum1997iterative,hackbusch2013multi}.

For any given one-level DD method (stationary or preconditioning), the choices of $V_c$, $P$ and $R$ influence
very strongly the convergence behavior of the corresponding two-level method. 
For this reason, the main focus of all the references mentioned above is the definition of different coarse spaces
and new strategies to build coarse space functions, leading to efficient two-level DD stationary and preconditioning
methods. Despite the mentioned references consider several one-level DD methods
and different partial differential equation (PDE) problems, it is still possible to classify them in two main groups.
These depend on the idea governing the definition of the coarse space.
To explain it, let us consider a DD iterative method (e.g., RAS) applied to a well-posed PDE problem.
Errors and residuals of the DD iterative procedure have generally very special forms. 
The errors are harmonic, in the sense of the underlying PDE operator, in the interior of the
subdomains (excluding the interfaces).
Moreover, the errors are predominant in the overlaps.
The residuals are predominant on the interfaces and zero outside the overlap.
For examples and more details, see, e.g., \cite{gander2017shem,CiaramellaGanderMamooler,CiaramellaRefl}.
This difference motivated, sometimes implicitly, the construction of different coarse spaces.
On the one hand, many references use different techniques to define coarse functions 
in the overlap (where the error is predominant), and then extending them on
the remaining part of the neighboring subdomains; see, e.g., 
\cite{Dohrmann2008,DoleanDTN,Galvis3,Galvis,Galvis2,Klawonn,Klawonn2015,Spillane2011,Spillane2014}.
On the other hand, in other works the coarse space is created by first defining basis function on the interfaces
(where the residual is non-zero), and then extend them (in different ways) on the portions of the neighboring subdomains;
see, e.g., \cite{Aarnes2002,Bjorstad2018,Chaouqui2018Coarse,Chaouqui2019,CiaramellaRefl,gander2014new,gander2015analysis,gander2017shem,
gander2019song,GanderVanC2019,Klawonn,GanderVanzan19}.
For a good, compact and complete overview of several of the different coarse spaces,
we refer to \cite[Section 5]{Klawonn}. For other different techniques and related discussions,
see, e.g., \cite{DoleanBook,Dubois2012,gander2014new,GHS2018,Graham2007,Zampini2017}.

The scenario is actually even more complicate, because different one-level DD methods are used (e.g., overlapping
and non-overlapping methods) and different PDEs are considered. However, the classifications we used so far
are sufficiently accurate to allow us to give a precise description of the novelties of our work.
We introduce for the first time so-called two-level DD substructured methods.
These are  two-level stationary iterative methods and
the term ``substructured'' indicates that iterations and coarse spaces are defined
on the interfaces (or skeletons).\footnote{Notice that the term ``substructured'' refers very often to DD methods
that are defined on non-overlapping subdomains; see, e.g., \cite{quarteroni1999domain,ToselliWidlund}.
However, in this work it indicates methods are purely defined on the interfaces,
independently of the type of (overlapping or non-overlapping) decomposition of the domain;
see, e.g., \cite[Section 5]{Gander1}.} 

With this respect, they are defined in the same spirit as two-level methods whose
coarse spaces are extensions in volume of interfaces basis functions.
Moreover, they share some similarities with the two-level methods designed in \cite{CiaramellaRefl} 
for the solution of PDEs on perforated domains.

We call our two-level substructured DD methods Spectral 2-level Substructured (S2S) methods, for which the coarse space
is obtained as the span of certain interface functions. 
A common choice would be to use a spectral coarse space, that is the span of
the dominant eigenfunctions of the one-level iteration operator $G:=M^{-1}N$.
 However, the S2S framework allows one to choose arbitrarily the coarse space functions,
as, e.g., the ones proposed in several papers as \cite{gander2015analysis,gander2017shem,gander2019song,Klawonn}.
Following the idea of correcting the `badly converging' modes of $G$, several papers 
proposed new, and in some sense optimal, coarse spaces.
In the context of domain decomposition methods, we refer, e.g., to 
\cite{gander2014new,GHS2018,gander2019song}, where efficient coarse spaces have been designed for
parallel, restricted additive and additive Schwarz methods. 
Fundamental results are presented in \cite{xu_zikatanov_2017}: for a symmetric and positive definite $A$,
it is proved that the coarse space of size $m$ that minimizes the energy norm of 
the two-level iteration operator is the exactly the spectral coarse space made by the first $m$ dominant eigenfunctions of $G$.
The sharp result of \cite{xu_zikatanov_2017} provides a concrete (optimal) choice
of $V_c$ minimizing the energy norm of the two-level operator. 
This minimum value is generally an upper bound for the asymptotic convergence factor.

The substructured operator $A$ considered in this paper is not necessarily symmetric.
As we will see in Section \ref{subs:globalCF}, 
coarse spaces different from the spectral one can lead to better convergence.
The S2S method, discussed in Section \ref{sec:Two-Level-methods} is capable to successfully 
accommodate (and generate numerically) different coarse spaces.
Convergence results are presented in Section \ref{sec:conv_S2S}, where the relations between
the kernel of the two-level operator, its contraction factor and the spectrum of the one-level operator $G$
are extensively discussed. These results are obtained by a novel analysis based on an 
infinite-matrix representation of the two-level operator. This analysis has a rather general applicability, it can be used to tackle non-symmetric problems, and allows us to show precisely in which cases a spectral coarse space is not (asymptotically) optimal.

From a numerical point of view, the S2S framework has several advantages if compared to a classical two-level DD method defined in volume.
Since the coarse space functions are defined on the interfaces, less memory storage is required. 
For a three-dimensional problem with mesh size $h$, a discrete interface coarse function is an array
of size $O(1/h^2)$. This is much smaller than $O(1/h^3)$, which is the size of an array corresponding
to a coarse function in volume. For this reason the resulting interface restriction and prolongation operators
are much smaller matrices, and thus the corresponding interpolation operations are cheaper to be performed.
Therefore, assuming that the one-level stationary iteration step and the dimension of the coarse space are 
the same for an S2S method and a method in volume, each S2S iteration is generally computationally less expensive.
In terms of iteration number, our S2S methods perform similarly or faster than
other two-level methods that use the same DD smoother.
Notice also, that the pre-computation part, that consists mainly in constructing
the coarse space $V_c$ and assembling the operators $P$, $R$ and $A_c$ requires the same computational effort
of a method in volume. Moreover, the substructured feature of the S2S framework allows us to
introduce two new procedures, based on a principal component analysis (PCA) and neural networks, 
for the numerical calculation of an efficient coarse space $V_c$.
Direct numerical experiments will show that the coarse spaces generated by these two approach
either outperform the spectral coarse space and other commonly used coarse spaces,
or they lead to a very similar convergence behavior.

This paper is organized as follows. In Section \ref{sec:Two-Level-method-Laplace}, we formulate
the classical parallel Schwarz method in a substructured form. This is done at the continuous level
and represents the starting point for the S2S method introduced in Section \ref{sec:Two-Level-methods}.
A detailed convergence analysis is presented in Section \ref{sec:conv_S2S}. Section \ref{sec:numerical_Vc} discusses both PCA-based and deep neural networks approaches to numerically create an efficient coarse space.
Extensive numerical experiments are presented in Section \ref{sec:num_exp}, where the robustness of the proposed
methods with respect to mesh refinement and physical (jumping) parameters is studied.
We present our conclusions in Section \ref{sec:conclusions}.
Finally, in the Appendix important implementation details are discussed.

\section{Substructured Schwarz domain decomposition methods}\label{sec:Two-Level-method-Laplace}
Consider a bounded Lipschitz domain $\Omega \subset \mathbb{R}^d$ for $d\in \{2,3\}$,
a general second-order linear elliptic operator $\mathcal{L}$ and a function $f \in L^2(\Omega)$.
Our goal is to introduce new domain-decomposition based methods for the efficient numerical solution 
of the general linear elliptic problem

\begin{equation}\label{model}
\mathcal{L} u = f \text{ in $\Omega$, $u = 0$ on $\partial \Omega$},
\end{equation}
which we assume to be uniquely solved by a $u \in H^1_0(\Omega)$.

To formulate our methods we need to fix some notation. 
Given a bounded set $\Gamma$ with boundary $\partial \Gamma$, we denote by $\rho_{\Gamma}(x)$ the
function representing the distance of $x \in \Gamma$ from $\partial \Gamma$. 
We can then introduce the $H_{00}^{1/2}(\Gamma)$ the space 
\begin{equation}\label{eq:LionsMagenes}
H_{00}^{1/2}(\Gamma) := \{ v \in H^{1/2}(\Gamma) \, : \, v/\rho_{\Gamma}^{1/2} \in L^2(\Gamma) \},
\end{equation}
which is also known as the Lions-Magenes space; see, e.g., \cite{lions1972non,quarteroni1999domain,tartar2007introduction}.
Notice that $H_{00}^{1/2}(\Gamma)$ can be equivalently defined as the space of
functions in $H^{1/2}(\Gamma)$ such that their extensions	 by zero to a superset
$\widetilde{\Gamma}$ of $\Gamma$ are in $H^{1/2}(\widetilde{\Gamma})$;
see, e.g., \cite{tartar2007introduction}.

Next, consider a decomposition of $\Omega$ into $N$ overlapping Lipschitz subdomains $\Omega_j$, that is
$\Omega = \cup_{j \in \mathcal{I}} \Omega_j$ with $\mathcal{I}:=\{1,2,\dots,N\}$.
For any $j \in \mathcal{I}$, we define the set of neighboring indexes
$\mathcal{N}_j :=\{ \ell \in \mathcal{I} \, : \, \Omega_j \cap \partial \Omega_\ell \neq \emptyset \}$.
Given a $j \in \mathcal{I}$, we introduce the substructure of $\Omega_j$ defined as
$\mathcal{S}_j := \cup_{\ell \in \mathcal{N}_j} \bigl(\Omega_j \cap \partial \Omega_\ell\bigr)$,
that is the union of all the portions of $\partial \Omega_\ell$ with $\ell \in \mathcal{N}_j$.\footnote{Notice that the substructure of a subdomain is sometimes called ``skeleton'';
see, e.g., \cite{CHS2}.}
Notice that the sets $\mathcal{S}_j$ are open and their closures are
$\overline{\mathcal{S}_j} = \mathcal{S}_j \cup \partial \mathcal{S}_j$,
with $\partial \mathcal{S}_j := \cup_{\ell \in \mathcal{N}_j} \bigl(\partial \Omega_j \cap  \partial \Omega_\ell \bigr)$.
Figure \ref{fig:decomposition} provides an illustration of substructures corresponding 
to a commonly used decomposition of a rectangular domain.
\begin{figure}
\centering
\begin{tikzpicture}[scale=0.15]
\draw[color=black] (9,19.5) node {$\Omega$};
\draw [thick,black] (0,0) rectangle (18,18);
  \draw [dashed,black] (0,6) -- (18,6); 
  \draw [dashed,black] (0,12) -- (18,12); 
  \draw [dashed,black] (6,0) -- (6,18); 
  \draw [dashed,black] (12,0) -- (12,18); 
  \draw [thick,black] (0,5) -- (18,5); 
  \draw [thick,black] (0,7) -- (18,7); 
  \draw [thick,black] (0,11) -- (18,11); 
  \draw [thick,black] (0,13) -- (18,13); 
  \draw[color=red] (9,9) node {$\Omega_j$};
  \draw [thick,black] (5,0) -- (5,18); 
  \draw [thick,black] (7,0) -- (7,18); 
  \draw [thick,black] (11,0) -- (11,18); 
  \draw [thick,black] (13,0) -- (13,18); 
  \draw [thick,red] (5,5) rectangle (13,13);
\begin{scope}[shift={(30,5)}]
\draw[color=black] (4,11.5) node {$\Omega_j$};
\draw[color=blue] (4,4) node {$\mathcal S_j$};
\draw [thick,red](-2,-2) rectangle (10,10);
  \draw [thick,blue] (-2,7) -- (10,7); 
  \draw [thick,blue] (1,-2) -- (1,10);
  \draw [thick,blue] (-2,1) -- (10,1);
  \draw [thick,blue] (7,-2) -- (7,10);    
\end{scope}
\end{tikzpicture}\caption{Decomposition of a rectangular $\Omega$ into nine overlapping subdomains (left), and representation of the substructure $\mathcal{S}_j$ for the central subdomain (right).}\label{fig:decomposition}
\end{figure}
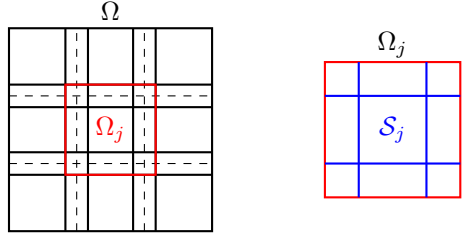
The substructure of $\Omega$ is defined as $\mathcal{S}:=\cup_{j \in \mathcal{I}}\overline{\mathcal{S}_j}$.
We denote by
$\mathcal{E}_j^0 : L^2(\mathcal{S}_j) \rightarrow L^2(\mathcal{S})$
the extension by zero operator.
Now, we consider a set of continuous functions $\chi_j : \overline{\mathcal{S}_j} \rightarrow [0,1]$,
$j=1,\dots,N$, such that
\begin{equation*}
\chi_j(x) \in \begin{cases}
(0,1] &\text{for $x \in \mathcal{S}_j$}, \\
\{1\} &\text{for $x \in \overline{\mathcal{S}_j \setminus \cup_{\ell \in \mathcal{N}_j} \mathcal{S}_\ell}$},\\
\{0\} &\text{for $x \in \partial \mathcal{S}_j \setminus \partial \Omega$},
\end{cases}
\end{equation*}
and $\sum_{j\in \mathcal{I}} \mathcal{E}_j^0 \chi_j \equiv1$, which means that the functions $\chi_j$
form a partition of unity.
Further, we assume that the functions $\chi_j$, $j \in \mathcal{I}$, satisfy the condition
$\chi_j / \rho_{\mathcal{S}_j}^{1/2}~\in~L^{\infty}(\mathcal{S}_j)$.

For any $j \in \mathcal{I}$, we define $\Gamma_j^{\Int} := \partial \Omega_j \cap \bigl( \cup_{ \ell \in \mathcal{N}_j} \Omega_\ell \bigr)$ and introduce the following trace and restriction operators 
$$\tau_j : H^1(\Omega_j) \rightarrow H^{1/2}(\mathcal{S}_j)
\text{ and } \tau_j^{\Int} : H^{1/2}(\mathcal{S}) \rightarrow H^{1/2}(\Gamma_j^{\Int}).
$$

It is well known that \eqref{model} is equivalent to the domain decomposition system (see, e.g., \cite{quarteroni1999domain})
\begin{equation}\label{modelsystem}
\begin{split}
\mathcal{L}u_j &= f_j \text{ in $\Omega_j$}, \:
u_j = \sum_{\ell \in \mathcal{N}_j} \mathcal{E}_\ell^0 (\chi_\ell \tau_\ell u_\ell)
\text{ on $\Gamma_j^{\Int}$}, \:
u_j = 0 \text{ on $\partial \Omega_j \setminus \Gamma_j^{\Int}$},
\end{split}
\end{equation}
where $f_j \in L^2(\Omega_j)$ is the restriction of $f$ on $\Omega_j$.
Notice that, since $\tau_\ell u_\ell \in H^{1/2}(\mathcal{S}_\ell)$, the properties of the
partition of unity functions $\chi_\ell$ guarantee that
$\chi_\ell \tau_\ell u_\ell$ lies in $H_{00}^{1/2}(\mathcal{S}_\ell)$
and $\mathcal{E}_\ell^0(\chi_\ell \tau_\ell u_\ell) \in H_{00}^{1/2}(\mathcal{S})$.
Moreover, for $\ell \in \mathcal{N}_j$ it holds that
$\tau_j^{\Int} \mathcal{E}_\ell^0(\chi_\ell \tau_\ell u_\ell) \in H_{00}^{1/2}(\Gamma_j^{\Int})$
if $\Gamma_j^{\Int} \subsetneq \partial \Omega_j$, and
$\tau_j^{\Int} \mathcal{E}_\ell^0(\chi_\ell \tau_\ell u_\ell) \in H^{1/2}(\Gamma_j^{\Int})$
if $\Gamma_j^{\Int} = \partial \Omega_j$.

Given a $j \in \mathcal{I}$ such that $\partial \Omega_j \setminus \Gamma_j^{\Int} \neq\emptyset$, we define the extension operator \linebreak $\mathcal{E}_j : H_{00}^{1/2}(\Gamma_j^{\Int}) \times L^2(\Omega_j) \rightarrow H^1(\Omega_j)$ as $w=\mathcal{E}_j(v,f_j)$, where $w$ solves the problem
\begin{equation}
\mathcal{L}w=f_j \text{ in $\Omega_j$}, \:
w=v \text{ on $\Gamma_j^{\Int}$}, \:
w=0 \text{ on $\partial\Omega_j\setminus\Gamma_j^{\Int}$}
\end{equation}
for $v \in  H_{00}^{1/2}(\Gamma_j^{\Int})$. Otherwise, if $\Gamma_j^{\Int} \equiv \partial \Omega_j$, we define $\mathcal{E}_j : H^{1/2}(\Gamma_j^{\Int}) \times L^2(\Omega_j) \rightarrow H^1(\Omega_j)$ as $w=\mathcal{E}_j(v,f_j)$, where $w$ solves the problem
\begin{equation}
\mathcal{L}w=f_j \text{ in $\Omega_j$}, \:
w=v \text{ on $\Gamma_j^{\Int}$}, \:
\end{equation}
for $v \in  H^{1/2}(\Gamma_j^{\Int})$.
The domain decomposition system \eqref{modelsystem} can be then written as
\begin{equation}\label{eq:u_j}
u_j = \mathcal{E}_j(0,f_j) + \mathcal{E}_j\Bigl(\tau_j^{\Int}\sum_{\ell \in \mathcal{N}_j} \mathcal{E}_\ell^0 (\chi_\ell \tau_\ell u_\ell),0\Bigr), \: j\in\mathcal{I}.
\end{equation}
If we define $v_j := \chi_j \tau_j u_j$, $j\in \mathcal{I}$, then system \eqref{eq:u_j} becomes
\begin{equation}\label{modelsubstructed}
v_j = g_j +\sum_{\ell \in \mathcal{N}_j}G_{j,\ell}(v_\ell), \: j\in\mathcal{I},
\end{equation}
where $g_j := \chi_j \tau_j\mathcal{E}(0,f_j)$ and the operators $G_{j,\ell} : H_{00}^{1/2}(\mathcal{S}_\ell) \rightarrow H_{00}^{1/2}(\mathcal{S}_j)$ are defined as
\begin{equation}\label{eq:Gj}
G_{j,\ell}(\cdot) := \chi_j \tau_j\mathcal{E}_j \bigl( \tau_j^{\Int} \mathcal{E}_\ell^0 (\cdot), 0 \bigr).
\end{equation}
System \eqref{modelsubstructed} is the substructured form of \eqref{modelsystem}.
The equivalence between \eqref{modelsystem} and \eqref{modelsubstructed} is explained by the following theorem.

\begin{theorem}[Equivalence between \eqref{modelsystem} and \eqref{modelsubstructed}]\label{thm:equivalence}
Let $u_j \in H^1(\Omega_j)$, $j\in \mathcal{I}$, solve \eqref{modelsystem},
then $v_{j} := \chi_j\tau_j(u_j)$, $j\in \mathcal{I}$, solve \eqref{modelsubstructed}.
Let $v_j \in H^{1/2}(\mathcal{S}_j)$, $j\in \mathcal{I}$, 
solve \eqref{modelsubstructed}, then $u_j := \mathcal{E}_j(\tau_j^{\Int}\sum_{\ell \in \mathcal{N}_j} \mathcal{E}_\ell^0 (v_\ell),f_j)$, $j\in \mathcal{I}$, solve \eqref{modelsystem}.
\end{theorem}
\begin{proof}
The first statement is proved before Theorem \ref{thm:equivalence}, where the substructured system \eqref{modelsubstructed} is derived.
To obtain the second statement, we use \eqref{modelsubstructed} and the definition of $u_j$ to write
$v_j = \chi_j \tau_j \mathcal{E}_j(\tau_j^{\Int} \sum_{\ell \in \mathcal{N}_j } \mathcal{E}_\ell^0 (v_\ell),f_j) = \chi_j \tau_j u_j$.
The claim follows by using this equality together with the definitions of $u_j$ and $\mathcal{E}_j$.
\end{proof}

Take any function $w \in H^1_0(\Omega)$ and consider the
initialization $u_j^0:=w|_{\Omega_j}$, $j \in \mathcal{I}$.
The parallel Schwarz method (PSM) is given by
\begin{equation}\label{PSMvolume}
\begin{split}
\mathcal{L}u_j^n &= f_j \text{ in $\Omega_j$}, \:
u_j^n = \sum_{\ell \in \mathcal{N}_j} \mathcal{E}_\ell^0 (\chi_\ell \tau_\ell u_\ell^{n-1})
\text{ on $\Gamma_j^{\Int}$}, \:
u_j^n = 0 \text{ on $\partial \Omega_j \setminus \Gamma_j^{\Int}$},
\end{split}
\end{equation}
for $n \in \mathbb{N}^+$, and has the substructured form
\begin{equation}\label{PSMsub}
v_j^n = g_j +\sum_{\ell \in \mathcal{N}_j}G_{j,\ell}(v_\ell^{n-1}), \: j \in \mathcal{I},
\end{equation}
initialized by $v_{j}^0 := \chi_j \tau_j(u_j^0) \in H_{00}^{1/2}(\mathcal{S}_j)$.
Notice that the iteration \eqref{PSMsub} is well posed in the sense that $v_{j}^n \in H_{00}^{1/2}(\mathcal{S}_j)$ for $j \in \mathcal{I}$ and $n \in \mathbb{N}$.
Equations \eqref{PSMsub} and \eqref{modelsubstructed} allow us to obtain the substructured PSM 
in error form, that is
\begin{equation}\label{PSMsuberr}
e_j^n = \sum_{\ell \in \mathcal{N}_j}G_{j,\ell}(e_\ell^{n-1}), \: j \in \mathcal{I},
\end{equation}
for $n \in \mathbb{N}^+$, where $e_{j}^n:=v_j-v_j^n$, for $j \in \mathcal{I}$ and $n\in \mathbb{N}$.
Equation \eqref{modelsubstructed} can be written in the matrix form $A{\bf v}={\bf b}$,
where ${\bf v}=[v_1,\dots,v_N]^\top$, ${\bf b}=[g_1,\dots,g_N]^\top$ and the entries of $A$ are
\begin{equation}\label{matrixexpression}
[A]_{j,j}=I_{d,j} \text{ and } [A]_{j,\ell}=-G_{j,\ell}, \: j,k\in \mathcal{I}, \: j \neq k,
\end{equation}
where $I_{d,j}$ are the identities on $L^2(\mathcal{S}_j)$, $j \in \mathcal{I}$.
Similarly, we define the operator $G$ as
\begin{equation*}
[G]_{j,j}=0 \text{ and } [G]_{j,\ell}=G_{j,\ell}, \: j,k\in \mathcal{I}, \: j \neq k,
\end{equation*}
and hence write \eqref{PSMsub} and \eqref{PSMsuberr} as ${\bf v}^n=G{\bf v}^{n-1}+{\bf b}$ and
${\bf e}^n=G{\bf e}^{n-1}$, respectively,
where ${\bf v}^n :=[v_1^n,\dots,v_N^n]^\top$ and ${\bf e}^n :=[e_1^n,\dots,e_N^n]^\top$.
Notice that $G=I-A$, where $I:=\text{diag}_{j=1,\dots,N}(I_{d,j})$.
Moreover, if we define 
$$\mathcal{H} := H_{00}^{1/2}(\mathcal{S}_1) \times \cdots \times H_{00}^{1/2}(\mathcal{S}_N),$$
then one can clearly see that $A : \mathcal{H} \rightarrow \mathcal{H}$
and $G : \mathcal{H} \rightarrow \mathcal{H}$.

It is a standard result that the PSM iteration ${\bf v}^n=G{\bf v}^{n-1}+{\bf b}$ converges;
see, e.g., \cite{CHS2} for a convergence result of the PSM in a substructured form,
\cite{Chaouqui2018Coarse,CiaramellaGander,CiaramellaGander2,CiaramellaGander3,DDLOGO,CiaramellaHoefer,CHS1,Gander1}
for other convergence results and \cite{DoleanBook,ToselliWidlund} for standard references. 
The corresponding limit is the solution to the problem $A{\bf v}={\bf b}$. 
From a numerical point of view, this is not necessarily true
if the (discretized) subproblems \eqref{PSMvolume} are not solved exactly.
For this reason, we assume in what follows that the subproblems \eqref{PSMvolume} are always solved exactly.

\section{S2S: Spectral two-level substructured DD method}\label{sec:Two-Level-methods}

The idea of the S2S method is to use a coarse space $V_c$ defined
as the span of certain linearly independent functions defined on the skeletons of
the subdomains $\Omega_j$, for $j \in \mathcal{I}$.
Consider the space $\mathcal{H}$, endowed with an inner product $\langle \cdot , \cdot \rangle$,
and a set of $m>0$ linearly independent functions $\psib_k$, $k=1,\dots,m$.
Notice that each $\psib_k$ has the form $\psib_k = [ \, \psi_k^1 , \dots , \psi_k^N \, ]^\top$,
where $\psi_k^j \in H_{00}^{1/2}(\mathcal{S}_j)$ for $j \in \mathcal{I}$.
We define the coarse space $V_c$ as
$$V_c := \spa \{ \psib_1 , \dots , \psib_m \}.$$
To define a two-level method, we need restriction and prolongation operators.
Once the coarse space $V_c$ is constructed, the choice of these operators follows naturally.
We define the prolongation operator $P : \mathbb{R}^{m} \rightarrow \mathcal{H}$ and 
the restriction operator $R : \mathcal{H} \rightarrow \mathbb{R}^{m}$ as
\begin{equation}\label{prol-restr_intro}
P {\bf v}:=
\sum\limits_{k=1}^{m} ({\bf v})_k\psib_k,\quad
\text{ and }\quad
R {\bf f} :=\begin{bmatrix}
\langle \psib_1, {\bf f}\rangle, &
\cdots, &
\langle \psib_m, {\bf f}\rangle
\end{bmatrix}^\top,
\end{equation} 
for any ${\bf v} \in \mathbb{R}^m$ and ${\bf f}  \in \mathcal{H}$.
Notice that, if the functions $\psib_k$ are orthogonal, $P$ is the adjoint operator of
$R$ and we have that $RP=I_m$, where $I_m$ is the identity
matrix in $\mathbb{R}^{m \times m}$.
The restriction of the operator $A$ on $V_c$ is the matrix $A_c \in \mathbb{R}^{m \times m}$
obtained in a Galerkin manner, $A_c = RAP$.

With the operators $P$, $R$ and $A_c$ in hands, our two-level method is defined as a classical
two-level strategy applied to the substructured problem \eqref{modelsubstructed}
and using the domain decomposition iteration \eqref{PSMsub} as a smoother.
This results in Algorithm \ref{two-level},
\begin{algorithm}[t]
\setlength{\columnwidth}{\linewidth}
\caption{Two-level substructured domain decomposition method}
\begin{algorithmic}[1]\label{two-level}
\REQUIRE ${\bf u}^{0}$ $\qquad \qquad \qquad \qquad \quad \:$ (initial guess)
\STATE ${\bf u}^{n}=G{\bf u}^{n-1} +{\bf b}$, $n=1,\dots,n_1$ (DD pre-smoothing steps)
\STATE ${\bf r} = {\bf b}-A{\bf u}^{n_1}$ $\qquad \qquad \qquad \quad \; \,$ (compute the residual)
\STATE Solve $A_c {\bf u}_c=R{\bf r}$ $\qquad \qquad \quad \; \; \; \,$ (solve the coarse problem)
\STATE ${\bf u}^{0} = {\bf u}^{n_1}+P{\bf u}_c$ $\qquad \qquad \qquad \, \, $ (coarse correction)
\STATE ${\bf u}^{n}=G{\bf u}^{n-1} +{\bf b}$, $n=1,\dots,n_2$ (DD post-smoothing steps)
\STATE Set ${\bf u}^{0}={\bf u}^{n_2}$ $\qquad \qquad \qquad \quad \; \; \,$ (update)
\STATE Repeat from 1 to 6 until convergence
\end{algorithmic}
\end{algorithm}
where $n_1$ and $n_2$ are the numbers of the pre- and post-smoothing steps.

The well posedness of Algorithm~\ref{two-level} is proved in the next lemma.
\begin{lemma}[Well posedness of S2S]\label{lemma:S2S_well}
Consider the inner product space $(\mathcal{H},\langle \cdot , \cdot \rangle)$,  
a set of linearly independent functions $\{ \psib_k \}_{k=1,\dots,m}$, for 
some $m>0$, and let 
$V_c := \spa\{ \psib_1 , \dots , \psib_m \}$ be a finite-dimensional subspace of $\mathcal{H}$.
Let $P$ and $R$ be defined as in \eqref{prol-restr_intro} (with $\langle \cdot , \cdot \rangle$).
If $A_c = RAP$ is invertible and
the initialization vector ${\bf u}^0$ is chosen in $\mathcal{H}$,
then ${\bf u}^{n_2}$ (computed at Step 5 of Algorithm \ref{two-level}) is in $\mathcal{H}$.
\end{lemma}
\begin{proof}
It is sufficient to show that for a given 
${\bf u}^0 \in \mathcal{H}$ all the steps of Algorithm~\ref{two-level}
are well posed. Since ${\bf b} \in \mathcal{H}$, $G : \mathcal{H} \rightarrow \mathcal{H}$
and $A : \mathcal{H} \rightarrow \mathcal{H}$, Step 1 and Step 2 produce ${\bf u}^{n_1}$ and
${\bf r}$ in $\mathcal{H}$.
Step 3 is well posed because $A_c$ is assumed to be invertible.
Since $V_c$ is a subset of $\mathcal{H}$,
$P {\bf u}_c$ and ${\bf u}^0$ in Step 4 lie in $\mathcal{H}$.
Clearly, the element ${\bf u}^{n_2}$ produced by Step 5 is also in $\mathcal{H}$.
Therefore, by induction we obtain that Algorithm \ref{two-level} is well posed in $\mathcal{H}$.
\end{proof}

The key hypothesis of Lemma \ref{lemma:S2S_well} is the invertibility of the coarse matrix
$A_c$. In Section \ref{subs:invAc}, an equivalent characterization of this property is proved.
This result (and the discussion thereafter) allows us to obtain the invertibility of $A_c$
if, e.g., $V_c$ is a spectral coarse space. Moreover, it is worth to remark that,
as discussed in \cite{xu_zikatanov_2017}, the pseudo-inverse of $A_c$ can be used 
in case $A_c$ is not invertible.

Let us now turn our attention to the coarse space $V_c$.
We distinguish two general classes of coarse space functions: global and local coarse functions.
Global coarse functions refer to functions defined directly on the global skeleton of $\Omega$.
An ideal choice of global coarse functions would be to define $V_c$ as the span of 
of the dominating eigenfunctions of the one-level operator $G$. 
In the context of multigrid methods, this choice is extensively discussed in \cite{xu_zikatanov_2017},
where the authors prove that, if $A$ and the preconditioner corresponding to the one-level
iteration are symmetric, the spectral coarse space minimizes the energy norm of $T$.
This sharp result provides a concrete optimal choice
of $V_c$ minimizing the energy norm of $T$. This is generally an upper bound for
the asymptotic convergence factor $\rho(T)$, as we will see in Section \ref{subs:globalCF}.
Moreover, we will show in Section \ref{sec:numerical_Vc} two numerical approaches, based on a PCA approach and neural networks, for the construction of global coarse space functions.
These are generally different from the spectral ones and lead to a better convergence.

Another possibility is to build local coarse functions using eigenfunctions of the local operators $G_j$.
However, the eigenfunctions of $G_j$ (or $G$) are known only in very special cases and their numerical computation
could be quite expensive. To overcome this problem one could define $V_c$ as the span of some Fourier basis
functions, that could be obtained by solving a Laplace-Beltrami eigenvalue problem on each interface (or skeleton);
see, e.g., \cite{Klawonn,gander2015analysis}.
In this case, assuming that local basis functions $\psi_k^j \in H_{00}^{1/2}(\mathcal{S}_j)$
(endowed with inner product $\langle \cdot , \cdot \rangle_j$) are available,
the coarse space $V_c$ can be constructed as
\[V_c := \spa_{j \in \mathcal{I} , k=1,\dots,\widehat{m}} \left\{ {\bf e}_j \otimes \psi_k^j \right\},\]
for some positive integer $\widehat{m}$, where $\otimes$ denotes the standard Kronecker product
and ${\bf e}_j$, for $j \in \mathcal{I}$, are the canonical vectors in $\mathbb{R}^N$.
In this case prolongation and restriction operators defined in \eqref{prol-restr_intro} are
\begin{equation}\label{prolungation-restriction_intro}
\begin{split}
P\begin{bmatrix}
{\bf v}^1\\
\vdots \\
{\bf v}^N\\
\end{bmatrix}
&=\begin{bmatrix}
\sum\limits_{k=1}^{\widehat{m}} ({\bf v}^1)_k\psi_k^1, 
& \cdots &
\sum\limits_{k=1}^{\widehat{m}} ({\bf v}^N)_k\psi_k^N
\end{bmatrix}^\top, \\
\quad
R\begin{bmatrix}
f_1\\
\vdots \\
f_N \\
\end{bmatrix}
&=\begin{bmatrix}
\langle \psi_1^1,f_1\rangle_1, &
\cdots, &
\langle \psi_{\widehat{m}}^1,f_1\rangle_1, &
\cdots &
\langle \psi_1^N,f_N\rangle_N, &
\cdots, &
\langle \psi_{\widehat{m}}^N,f_N\rangle_N
\end{bmatrix}^\top,
\end{split}
\end{equation} 
for any ${\bf v}^1,\dots,{\bf v}^N \in \mathbb{R}^{\widehat{m}}$ and any $(f_1,\dots,f_N) \in \mathcal{H}$.
We wish to remark, that the choice of the inner product $\langle \cdot , \cdot \rangle$
(or $\langle \cdot , \cdot \rangle_j$ for $j\in \mathcal{I}$) in the definition of $P$ and $R$ is arbitrary.
One possible choice is the classical $H^{1/2}$ inner product. However, this could be
too expensive from a numerical point of view. Another possibility would be to consider
the classical $L^2$ inner product, which is the choice we make in our implementations.

A detailed convergence analysis that covers our S2S method, is presented in Section \ref{sec:conv_S2S}.
This is based on the general structure of a two-level iteration operator.
A direct calculation reveals that one iteration of the S2S method can be written as 
\begin{equation}\label{eq:stationary}
{\bf u}^{\rm new}=G^{n_2}(I-PA_c^{-1}RA)G^{n_1}{\bf u}^{\rm old} + \widetilde{M} {\bf b},
\end{equation}
where $I$ is the identity operator over $\mathcal{H}$;
see, also, \cite{CiaramellaRefl,gander2014new,hackbusch2013multi}.
Here, $\widetilde{M}$ is an operator which acts on the right-hand side vector ${\bf b}$.
Such operator can be regarded as the preconditioner corresponding to our two-level method.
In error form, the iteration \eqref{eq:stationary} becomes
\begin{equation}\label{twolevelrecurrence}
{\bf e}^{\rm new}=T{\bf e}^{\rm old} \text{ with } T:=G^{n_2}(I-PA_c^{-1}RA)G^{n_1},
\end{equation}
where ${\bf e}^{\rm new}:= {\bf u}-{\bf u}^{\rm new}$ and ${\bf e}^{\rm old}:= {\bf u}-{\bf u}^{\rm old}$.
Hence, to prove convergence of the S2S method we study the operator $T$.
For simplicity most of the results will be proved for $n_1=1$ and $n_2=0$.

\section{Convergence analysis}\label{sec:conv_S2S}
In this section, we provide convergence results for two-level iterative methods in
a general framework that covers the setting of the S2S domain decomposition method
presented in Section \ref{sec:Two-Level-methods}.

Let $(\mbX, \langle \cdot , \cdot \rangle)$ be a complex\footnote{The hypothesis of a complex inner-product space is general and has the goal of dealing with possibly complex eigenvectors of non-symmetric 
$\mbA$ and $\mbG$. Nevertheless, the analysis presented in this section is valid also in the (more commonly used) real case.} inner-product space and
$\mbA x = b$ a linear problem, where the operator $\mbA : \mbX \rightarrow \mbX$ 
is bijective and $b \in \mbX$ is a given vector. 
Consider a set of $m>0$ linearly independent functions $\{\psib_k\}_{k=1,\dots,m}$,
and denote by $V_c$ the finite-dimensional subspace of $\mathcal{X}$ defined as the span of the functions
$\{\psib_k\}_{k=1,\dots,m}$. We denote by $P : \C^m \rightarrow \mbX$ and $R : \mbX \rightarrow \C^m$ the
prolongation and restriction operators defined as in \eqref{prol-restr_intro},
and define the matrix $\mbA_c := R \mbA P \in \C^{m \times m}$.
Given a smoothing operator $\mbG : \mbX \rightarrow \mbX$, a two-level iterative method
(as the one defined in Algorithm \ref{two-level}) is characterized by the iteration operator
$\mbT : \mbX \rightarrow \mbX$ defined by
\begin{equation}\label{eq:itT}
\mbT := \mbG^{n_2} ( \mbI - P \mbA_c^{-1} R \mbA ) \mbG^{n_1},
\end{equation}
where $\mbI : \mbX \rightarrow \mbX$ is the identity operator.
In what follows the properties of $\mbT$ are analyzed. In particular, 
the invertibility of $\mbA_c$ is characterized in Section \ref{subs:invAc},
the convergence (spectral) properties of $\mbT$ are discussed in the case
of global coarse functions in Section \ref{subs:globalCF} and in the case of
local coarse functions in Section \ref{subs:localCF}.

\subsection{Invertibility of the coarse matrix}\label{subs:invAc}

The well-posedness of a two-level method (like the S2S) is essentially related to the invertibility of the coarse operator $\mbA_c$. Even though one could replace the inverse
of $\mbA_c$ with its pseudo-inverse, as discussed in, e.g., \cite{xu_zikatanov_2017}, 
in our analysis we will assume that $\mbA_c$ is invertible. 
The next Lemma provides an equivalent characterization for the invertibility of $\mbA_c$.

\begin{lemma}[Invertibility of a coarse operator $\mbA_c$]\label{lemma:invAc}
Let $\mathbb{P}_{V_c} : \mbX \rightarrow V_c$ be the projection operator onto $V_c$. 
The coarse matrix $\mbA_c = R \mbA P$ has full rank if and only if 
$\mathbb{P}_{V_c}( \mbA {\bf v} ) \neq 0 \: \forall {\bf v} \in V_c\setminus \{0\}$.
\end{lemma}
\begin{proof}
We first show that if $\mathbb{P}_{V_c}( \mbA{\bf v} ) \neq 0$ for any ${\bf v} \in V_c\setminus \{0\}$,
then $\mbA_c = R\mbA P$ has full rank.
This result follows from the rank-nullity theorem, if we show that the only element in the kernel
of $\mbA_c$ is the zero vector. To do so, we recall the definitions of $P$ and $R$ given in
\eqref{prol-restr_intro}. Let us now consider a vector ${\bf z}\in \C^m$ .
Clearly, $P {\bf z}=0$ if and only if ${\bf z}=0$.
Moreover, for any ${\bf z} \in \C^m$ the function $P {\bf z}$ is in $V_c$.
Since $\mbA$ is invertible, then $\mbA P{\bf z}=0$ if and only if ${\bf z}=0$.
Moreover, by our assumption it holds that $\mathbb{P}_{V_c}(AP{\bf z}) \neq 0$.
Now, we notice that $R{\bf w} \neq 0$ for all ${\bf w} \in V_c \setminus\{0\}$,
and $R{\bf w} = 0$ for all ${\bf w} \in V_c^\perp$, where $V_c^\perp$ denotes 
the orthogonal complement of $V_c$ in $\mathcal{X}$ with respect to $\langle \cdot , \cdot \rangle$. 
Since $(\mathcal{X},\langle \cdot , \cdot \rangle)$ is an inner-product space,
we have $\mbA P{\bf z} = \mathbb{P}_{V_c}(\mbA P{\bf z}) + (\mbI-\mathbb{P}_{V_c})(\mbA P{\bf z})$
with $(\mbI-\mathbb{P}_{V_c})(\mbA P{\bf z}) \in V_c^\perp$.
Hence, $R\mbA P{\bf z} = R\mathbb{P}_{V_c}(\mbA P{\bf z}) \neq 0$
for any non-zero ${\bf z}$.

Now we show that, if $\mbA_c = R\mbA P$ has full rank, then $\mathbb{P}_{V_c}( \mbA {\bf v}  ) \neq 0$ for any
${\bf v} \in V_c\setminus \{0\}$. We proceed by contraposition and prove that if there exists
a ${\bf v} \in V_c\setminus \{0\}$ such that $\mbA {\bf v} \in V_c^\perp$, 
then $\mbA_c = R\mbA P$ has not full rank.
Assume that there is a ${\bf v} \in V_c\setminus \{0\}$ such that $\mbA {\bf v} \in V_c^\perp$.
Since ${\bf v}$ is in $V_c$, there exists a nonzero vector ${\bf z}$ such
that ${\bf v}=P{\bf z}$. Hence $\mbA P{\bf z} \in V_c^\perp$.
We can now write that $\mbA_c {\bf z}= R(\mbA P{\bf z})=0$,
which implies that $\mbA_c$ has not full rank.
\end{proof}

The following example shows that the invertibility of $A$ does not necessarily implies
the invertibility of $\mbA_c$.

\begin{example}

Consider the invertible matrix $\mbA:=\begin{footnotesize}
\begin{bmatrix} 0 & 1 \\ 1 & 0 \end{bmatrix} 
\end{footnotesize}$.
Let us denote by ${\bf e}_1$ and ${\bf e}_2$ the canonical vectors in $\mathbb{R}^2$,
define $V_c := \spa \{ {\bf e}_1 \}$, and consider the classical scalar product for $\mathbb{R}^2$.
This gives $V_c^\perp := \spa \{ {\bf e}_2 \}$.
The prolongation and restriction operators are $P={\bf e}_1$ and $R=P^\top$.
Clearly, we have that $\mbA {\bf e}_1 = {\bf e}_2$, which implies that
$\mathbb{P}_{V_c}( \mbA {\bf v} ) = 0$ for all ${\bf v} \in V_c$.
Moreover, in this case we get $\mbA_c = R\mbA P = 0$, which shows that $\mbA_c$ is not invertible.

\end{example}

Notice that, if $\mbA (V_c) \subseteq V_c$, then it holds that 
$\mathbb{P}_{V_c}( \mbA {\bf v} ) \neq 0$ $\forall {\bf v} \in V_c \setminus \{ 0\}$, and $\mbA_c$ is invertible. 
The condition $\mbA (V_c) \subseteq V_c$ is satisfied for operators of the form $\mbA = \mbI - \mbG$, as for instance those defined in \eqref{matrixexpression},
if the functions $\psib_k$ are eigenfunctions of $\mbG$. However, it represents only a
sufficient condition for the invertibility of $\mbA_c$.
As the following example shows, there exist invertible operators $\mbA$ that do not satisfy this condition,
but lead to invertible $\mbA_c$.

\begin{example}
Consider the invertible matrix $\mbA:=\begin{footnotesize}
\begin{bmatrix} 1 & 0 & 0 \\ 0 & 1 & 1 \\ 0 & 1 & 0 \end{bmatrix} 
\end{footnotesize}$.
Let us denote by ${\bf e}_1$, ${\bf e}_2$ and ${\bf e}_3$ the three canonical vectors in $\mathbb{R}^3$,
define $V_c := \spa \{ {\bf e}_1 , {\bf e}_2 \}$, and consider the classical scalar product for $\mathbb{R}^3$.
This gives $V_c^\perp := \spa \{ {\bf e}_3 \}$.
The prolongation and restriction operators are $P=[{\bf e}_1 , {\bf e}_2]$ and $R=P^\top$,
and we get $\mbA_c = R\mbA P = I$, where $I$ is the $2\times2$ identity matrix.
Now, we notice that $\mbA {\bf e}_2 = {\bf e}_2 + {\bf e}_3$, which implies that
$\mathbb{P}_{V_c}( \mbA {\bf e}_2 )\neq 0$ and $\mathbb{P}_{V_c^\perp}( \mbA {\bf e}_2 )\neq 0$.
Hence $V_c$ is not invariant under $\mbA$, but $\mbA_c$ is invertible.
\end{example}

\subsection{Global coarse functions}\label{subs:globalCF}
In this section, we study general convergence properties of the operator $\mbT$.
The first theorem characterizes the relation between
the kernel of $\mbT$ and the coarse space $V_c$.

\begin{theorem}[Kernel of $\mbT$, coarse space $V_c$]\label{thm:eqiv1}
Let $P$ and $R$ be defined as in \eqref{prol-restr_intro} by linearly independent
functions $\psib_1,\dots,\psib_m$ such that $\mbA_c = R\mbA P$ is invertible.
For any $\psib \in \mbX$ it holds that
\begin{equation}\label{eq:equiv:1}
[ \mbI - P \mbA_c^{-1} R \mbA ] \psib = 0 
\; \Leftrightarrow \; 
\psib \in V_c  := \spa\{ \psib_1,\dots,\psib_m \} .
\end{equation}
\end{theorem}

\begin{proof}
%
Assume that $\psib \in V_c$. 
This implies that there exists a vector ${\bf z}$ such that $\psib = P {\bf z}$.
Hence, we can compute
\begin{equation*}
[ \mbI - P \mbA_c^{-1} R \mbA ] \psib = \psib - P \mbA_c^{-1} R \mbA \psib
= P {\bf z} - P \mbA_c^{-1} R \mbA P{\bf z}=P{\bf z}-P{\bf z}=0. \\
\end{equation*}


Let us now prove the reverse, that is $[ \mbI - P \mbA_c^{-1} R \mbA ] \psib = 0 \Rightarrow \psib \in V_c$.
We proceed by contraposition and assume that $\psib \notin V_c$,
that is there exists a nonzero $\psib_b \in V_c^\perp$
such that $\psib = \psib_a + \psib_b$ with $\psib_a \in V_c$.
Since $\psib_a \in V_c$, we already know that $[ \mbI - P \mbA_c^{-1} R \mbA ] \psib_a = 0$. Hence, it holds that
\begin{equation}\label{eq:proofB}
[ \mbI - P \mbA_c^{-1} R \mbA ] \psib = [ \mbI - P \mbA_c^{-1} R \mbA ] (\psib_a + \psib_b) = \underbrace{\psib_b}_{ \in V_c^\perp }-\underbrace{P \mbA_c^{-1} R \mbA \psib_b}_{ \in V_c }\neq 0.
\end{equation}
\end{proof}

To continue our analysis we construct a matrix representation of the operator $\mbT$.
For this purpose, we consider the following assumptions:
\begin{itemize}[leftmargin=10mm]\itemsep0em
\item[(H1)] $V_c$ is the span of $m$ linearly independent functions $\{ {\bf p }_k\}_{k=1}^m \subset \mbX$,
which are used to define the operators $P$ and $R$ as in \eqref{prol-restr_intro}.
\item[(H2)] The operators $\mbA$ and $\mbG$ have the same linearly independent eigenvectors $\{ \psib_k \}_{k=1}^\infty$,
The corresponding eigenvalues of $\mbA$ and $\mbG$ are denoted by $\wlambda_k$ and $\lambda_k$, respectively.
\item[(H3)] The eigenvalues $\lambda_k$ satisfy $| \lambda_k |\in(0,1)$, $| \lambda_k| \leq |\lambda_{k-1}|$ for all $k$.
\item[(H4)] There exists an index $\wm \geq m$ such that $V_c$ satisfies the relations 
\begin{equation*}
V_c \subseteq {\rm span}\,\{ \psib_k \}_{k=1}^{\wm} \quad \text{and} \quad
V_c \cap {\rm span}\,\{ \psib_k \}_{k=\wm+1}^\infty=\{0\}.
\end{equation*}
\end{itemize}
\begin{remark}
Notice that the hypothesis (H2) is valid in the context of our S2S method,
where the operators $\mbA$ and $\mbG$ satisfy the relation $\mbA = \mbI - \mbG$.
Hence, they have the same eigenvectors. Moreover, the hypothesis (H3) is satisfied
if $\mbG$ corresponds to a classical parallel Schwarz method, as in the case of our S2S method. The (discrete) classical damped Jacobi method is another important instance that satisfies (H2) and (H3).
\end{remark}
Let us now construct a matrix representation of the operator $\mbT$.
Since $V_c \subseteq {\rm span}\,\{ \psib_k \}_{k=1}^{\wm}$, the structure of $\mbT$
allows us to obtain that the set ${\rm span}\,\{ \psib_k \}_{k=1}^{\wm}$
is invariant, that is $\mbT \psib_j \in {\rm span}\,\{ \psib_k \}_{k=1}^{\wm}$
for any $j=1,\dots,\wm$.
Similarly, a direct calculation reveals that $\mbT \psib_j = \lambda_j \psib_j - \sum_{\ell=1}^{\wm} x_{j-\wm,\ell} \psib_\ell$ for $j \geq \wm+1$ for some coefficients $x_j$.
Therefore, for any $\psib_j$ there exist at most $\wm+1$ nonzero coefficients
$\wt_{j,\ell}$ such that
$\mbT \psib_j = \wt_{j,j} \psib_j + \sum_{\ell=1,\ell\neq j}^{\wm} \wt_{j,\ell}\psib_\ell$. 
If we order the coefficients $\wt_{j,\ell}$ into an infinite matrix denoted by $\wT$,
we obtain that 
\begin{equation}\label{Ciaramella_mini_10_eq:wTT}
\begin{aligned}[c]
\wT = \begin{bmatrix}
\wT_{\wm} & 0 \\
X & \Lambda_{\wm}\\
\end{bmatrix},
\end{aligned}
\qquad
\begin{aligned}[c]
&\Lambda_{\wm} = {\rm diag}\, (\lambda_{\wm+1},\lambda_{\wm+2},\dots), \\
& \wT_{\wm} \in \mathbb{C}^{\wm \times \wm}, [X]_{j,\ell} = x_{j,\ell}, \ell=1,\dots,\wm, j=1,2,\dots
\end{aligned}
\end{equation}
The infinite matrix $\wT$ can be regarded as a linear operator acting
on the space of sequences. 
The matrix representation \eqref{Ciaramella_mini_10_eq:wTT} turns to be very useful to 
analyze the convergence properties of the operator $\mbT$.
Now, we can compute by an induction argument that
\begin{equation}\label{eq:Tn}
\wT^n = 
\begin{bmatrix}
\wT_{\wm}^n & 0 \\
P_n & \Lambda_{\wm}^n\\
\end{bmatrix}
\quad \text{with} \quad
P_n = \sum_{j=1}^n \Lambda_{\wm}^{n-j} X \wT_{\wm}^{j-1}.
\end{equation}
If the matrix $\wT_{\wm}$ is nilpotent with degree $q \in \mathbb{N}_+$,
that is $\wT_{\wm}^p=0$ for all $p\geq q$, then we get for $n>q$ that
$$P_n = \sum_{j=1}^q \Lambda_{\wm}^{n-j} X \wT_{\wm}^{j-1}
+ \sum_{j=q+1}^n \Lambda_{\wm}^{n-j} X \wT_{\wm}^{j-1}
= \Lambda_{\wm}^n \sum_{j=1}^{q} \Lambda_{\wm}^{-j}X \wT_{\wm}^{j-1}.$$
Thus, by defining $X_{q} := \sum_{j=1}^{q} \Lambda_{\wm}^{-j}X \wT_{\wm}^{j-1}$,
one gets for $n>q$ that
\begin{equation}\label{eq:bobo}
\wT^n = T_a^{n-1} T_b
\quad \text{with} \quad
T_a :=
\begin{bmatrix}
0 & 0 \\
0 & \Lambda_{\wm}
\end{bmatrix},
T_b :=
\begin{bmatrix}
0 & 0 \\
\Lambda_{\wm} X_{q} & \Lambda_{\wm}\\
\end{bmatrix}.
\end{equation}

Let us begin with a case where the linear operators $\mbA$ and $\mbG$ are bounded and self-adjoint,
and the functions $\{ \psib_k \}_{k=1}^\infty$
form an orthonormal basis with respect to an inner product $\llangle \cdot , \cdot \rrangle$
(not necessarily equal to $\langle \cdot , \cdot \rangle$)
such that $(\mbX,\llangle \cdot , \cdot \rrangle)$ is a Hilbert space.
We denote by $\| \cdot \|_{\mbX}$ the norm induced by $\llangle \cdot , \cdot \rrangle$,
and by
\begin{equation}\label{eq:op_normX}
\|S\|_{\mbX}:=\sup\limits_{\|{\bf v}\|_{\mbX}=1}\|S{\bf v}\|_{\mbX}
\text{ for any $S \in \mathcal{L}(\mbX)$},
\end{equation}
the corresponding operator norm.
Notice that, since $\mbA$ and $\mbG$ are bounded, $\mbT$ is bounded as well.
Thus, we can study the asymptotic convergence factor $\rho(\mbT)$ defined as
$\lim\limits_{n\rightarrow \infty} \|\mbT^n\|_{\mbX}^{1/n}=\rho(\mbT)$;
see, e.g., \cite[Chapter 17]{lax2002functional}.
Since we assumed that $\{ \psib_k \}_{k=1}^\infty$ are orthonormal with respect
to $\llangle \cdot , \cdot \rrangle$, a direct calculation\footnote{
$ \| \mbT \|_{\mbX}^2
=\sup\limits_{\|{\bf v}\|_{\mbX}=1}\| \mbT {\bf v} \|_{\mbX}^2
=\sup\limits_{\|{\bf w}\|_{\ell^2}=1}\| \mbT \sum\limits_j {\bf w}_j \psib_j\|_{\mbX}^2
=\sup\limits_{\|{\bf w}\|_{\ell^2}=1} \sum\limits_{j,k,i,p} {\bf w}_j \overline{{\bf w}_i} \wt_{j,k}
\overline{\wt_{i,p}} \llangle \psib_k , \psib_p \rrangle
=\sup\limits_{\|{\bf w}\|_{\ell^2}=1} \sum\limits_{j,k,i} {\bf w}_j \overline{{\bf w}_i} \wt_{j,k}
\overline{\wt_{i,k}} 
= \sup\limits_{\|{\bf w}\|_{\ell^2}=1} \| \wT \|_{\ell^2}^2.
$
} 
allows one to prove that
$\|\mbT\|_{\mbX} = \| \wT \|_{\ell^2}$, where
\begin{equation}\label{eq:op_normL2}
\| S \|_{\ell^2}:=\sup\limits_{\|{\bf v}\|_{\ell^2}=1}\|S{\bf v}\|_{\ell^2}
\text{ for any $S \in \mathcal{L}(\ell^2)$}.
\end{equation}
Hence, we obtain
$\rho(\mbT) = \lim\limits_{n\rightarrow \infty} \|\mbT^n\|_{\mbX}^{1/n}
= \lim\limits_{n\rightarrow \infty} \|\wT^n\|_{\ell^2}^{1/n}$.
Notice that since $\mbT$ is a bounded operator and $\|\mbT\|_{\mbX} = \| \wT \|_{\ell^2}$, the operator $\wT$ is bounded in the $\|\cdot\|_{\ell^2}$ norm. 
Thus, the submatrices $X$, $\Lambda_{\wm}$ and $\wT_{\wm}$
are bounded in the $\|\cdot\|_{\ell^2}$ norm as well.
Therefore, $T_a$ and $T_b$ are also
bounded in the $\|\cdot\|_{\ell^2}$ norm.
Thus, equation \eqref{eq:bobo} allows us to estimate $\rho(\mbT)$:
\begin{equation*}
\begin{split}
\rho(\mbT)&= \lim\limits_{n\rightarrow \infty} \|\wT^n\|_{\ell^2}^{1/n} 
\leq \lim\limits_{n\rightarrow \infty} (\|T_a^{n-1}\|_{\ell^2} \|T_b\|_{\ell^2})^{1/n}\leq \lim\limits_{n\rightarrow \infty} |\lambda_{\wm+1}|^{(n-1)/n} \|T_b\|_{\ell^2}^{1/n} \\
&= |\lambda_{\wm+1}|.
\end{split}
\end{equation*}
Now, recalling \eqref{eq:op_normL2}, one obtains for $n>q$ that
$$
\| \wT^n \|_{\ell^2}=\sup\limits_{\|{\bf v}\|_{\ell^2}=1}\|\wT^n{\bf v}\|_{\ell^2}
\geq \|\wT^n{\bf e}_{\wm+1}\|_{\ell^2} = |\lambda_{\wm+1}|^n,
$$
where ${\bf e}_{\wm+1} \in \ell^2$ is the $\wm+1$-th canonical vector.
This estimate implies that $\rho(\mbT) = \lim\limits_{n\rightarrow \infty} \|\wT^n\|_{\ell^2}^{1/n} \geq |\lambda_{\wm+1}|$, and thus
$\rho(\mbT) = |\lambda_{\wm+1}|$.
Using Theorem 2, it is possible to see that
the matrix $\wT_{\wm}$ is nilpotent with degree $q=1$, 
if $V_c = {\rm span}\,\{ \psib_k \}_{k=1}^{\wm}$.
In this case $|\lambda_{\wm+1}|=|\lambda_{m+1}|$.
We can summarize these findings in the next theorem.

\begin{theorem}[Convergence of a two-level method]\label{thm:conv:orth}
Let the hypotheses (H1), (H2), (H3) and (H4) be satisfied, $\mbA$ and $\mbG$ be self-adjoint, and assume that the functions $\{ \psib_k \}_{k=1}^\infty$
form an orthonormal basis with respect to an inner product $\llangle \cdot , \cdot \rrangle$
such that $(\mbX,\llangle \cdot , \cdot \rrangle)$ is a Hilbert space.
If $\wT_{\wm}$ is nilpotent (e.g., if
$V_c = {\rm span}\,\{ \psib_k \}_{k=1}^{\wm}$), then
$$
\rho(\mbT) =  \lim\limits_{n\rightarrow \infty} \|\mbT^n\|_{\mbX}^{1/n} = | \lambda_{\wm+1} | < 1,
$$
where $\|\cdot\|_{\mbX}$ is the operator norm defined in \eqref{eq:op_normX}.
\end{theorem}


In the case of a spectral coarse space, the expression of $\wT$ in \eqref{Ciaramella_mini_10_eq:wTT} simplifies. The following result holds.
\begin{theorem}[The matrix $\wT$ for self-adjoint $\mbA$ and $\mbG$ and a spectral coarse space]\label{thm:NEW}
Let the hypotheses (H1), (H2), (H3) and (H4) be satisfied.
If the functions $\{ \psib_k \}_{k=1}^\infty$ form an orthonormal basis for
$(\mbX,\llangle \cdot, \cdot \rrangle)$, 
the operators $\mbA$ and $\mbG$ are self adjoint, 
and $V_c = {\rm span}\,\{ \psib_k \}_{k=1}^{m}$ with $\wm=m$, then
\begin{equation*}
\wT = \begin{bmatrix}
0 & 0 \\
0 & \Lambda_{\wm}\\
\end{bmatrix},
\end{equation*}
where $\Lambda_{\wm}$ is defined in \eqref{Ciaramella_mini_10_eq:wTT}.
\end{theorem}
\begin{proof}
Since $V_c = {\rm span}\,\{ \psib_k \}_{k=1}^{m}$, Theorem \ref{thm:eqiv1}
implies that $\wT_{\wm}=0$.
Thus, to obtain the result, it is sufficient to show that all the components of the
submatrix $X$ (see \eqref{Ciaramella_mini_10_eq:wTT}) are zero.
These components are $x_{j,\ell} = \wt_{j,\ell}$ for $j>\wm$ and $\ell\leq \wm$.
Thus, we assume that  $j>\wm$ and $\ell\leq \wm$, recall the formula
$\mbT \psib_j = \wt_{j,j} \psib_j + \sum_{k=1,k\neq j}^{\wm} \wt_{j,k}\psib_k$,
and multiply this by $\psib_\ell$ to obtain
$\llangle \psib_\ell , \mbT \psib_j \rrangle = \wt_{j,\ell}$.
Since $\mbA$ and $\mbG$ are self adjoint, one obtains by a direct
calculation that
$[ \mbI - P \mbA_c^{-1} R \mbA ]^* = [ \mbI - \mbA P \mbA_c^{-1} R ]$.
Using this property and recalling the structure of $\mbT$, we can compute
\begin{equation*}
\begin{split}
\wt_{j,\ell} &= \llangle \psib_\ell , \mbT \psib_j \rrangle
= \llangle \psib_\ell , [ \mbI - P \mbA_c^{-1} R \mbA ] \mbG \psib_j \rrangle 
= \lambda_j \llangle\,  [ \mbI - P \mbA_c^{-1} R \mbA ]^*  \psib_\ell ,  \psib_j \rrangle\\
&=\lambda_j \llangle\,  [ \mbI - \mbA P \mbA_c^{-1} R ]  \psib_\ell ,  \psib_j \rrangle.
\end{split}
\end{equation*}
Now, since $[ \mbI - \mbA P \mbA_c^{-1} R ]  \psib_\ell \in {\rm span}\,\{ \psib_k \}_{k=1}^{\wm}$
and $\ell>\wm$,
the orthogonality of the functions $\{ \psib_k \}_{k=1}^\infty$ and the hypothesis
(H4) imply that
$\llangle\,  [ \mbI - \mbA P \mbA_c^{-1} R ]  \psib_\ell ,  \psib_j \rrangle = 0$.
Hence, the result follows.
\end{proof}
Theorem \ref{thm:NEW} implies directly that
$$\| \mbT \|_\mbX = \rho(\mbT) = |\lambda_{m+1}|.$$
Let us now assume that $\mbA$ is positive definite, and thus there exists a unique positive square root operator $\mbA^{1/2}$ such that $A^{1/2}\psib_j=\wlambda_j^{1/2}\psib_j$, \cite[Theorem 6.6.4]{friedman1982foundations}.
A straight calculation leads to $\|S\|_{\mbA} = \| \mbA^{1/2} S \mbA^{-1/2} \|_{\mbX}$ (see, e.g.,
\cite[Section C.1.3]{Hackbusch_book} for a finite-dimensional matrix counterpart).
Notice that, as for $\mbT$ and $\wT$, we can obtain the matrix representation
$\widetilde{\Lambda}^{1/2} \wT \widetilde{\Lambda}^{-1/2} $
of $\mbA^{1/2} \mbT \mbA^{-1/2}$, where $\wT$ is defined in \eqref{Ciaramella_mini_10_eq:wTT}
and $\widetilde{\Lambda}={\rm diag}\, (\wlambda_{1},\wlambda_{2},\dots)$.
Thus, as for $\| \mbT \|_{\mbX} = \| \wT \|_{\ell^2}$,
one can prove that $\| \mbA^{1/2} \mbT \mbA^{-1/2} \|_{\mbX} = \| \widetilde{\Lambda}^{1/2} \wT \widetilde{\Lambda}^{-1/2} \|_{\ell^2}$.
Hence, we get $\| \mbT \|_{\mbA} = \| \widetilde{\Lambda}^{1/2} \wT \widetilde{\Lambda}^{-1/2} \|_{\ell^2}$.
Now, if it holds that $V_c = {\rm span}\,\{ \psib_k \}_{k=1}^{m}$, then
Theorem \ref{thm:NEW} implies that
$$
\| \mbT \|_{\mbA} 
= \| \widetilde{\Lambda}^{1/2} \wT \widetilde{\Lambda}^{-1/2} \|_{\ell^2} 
= \| \Lambda_{\wm} \|_{\ell^2}
= | \lambda_{m+1} | = \rho(\mbT).
$$
It has been proved in \cite[Theorem 5.5]{xu_zikatanov_2017}, that this result is optimal in the sense that,
if $\mbA$ and $\mbG$ are symmetric and positive (semi-)definite, then the coarse space 
$V_c = {\rm span}\,\{ \psib_k \}_{k=1}^{m}$ minimizes the energy norm of the two-level 
operator $\mbT$. Clearly, if $\mbA$ has positive and negative eigenvalues (even though it remains symmetric),
this result is no longer valid. In this case, as we are going to see in Theorem \ref{Ciaramella_mini_10_thm:perturb},
the coarse space $V_c = {\rm span}\,\{ \psib_k \}_{k=1}^{m}$ is not necessarily 
(asymptotically) optimal.

The situation is very different if the functions $\{ \psib_k \}_{k=1}^\infty$ are not orthogonal
and $\mbA$ is not symmetric. To study this case, we work in a finite-dimensional setting
and assume that $\mbX = \mathbb{C}^N = {\rm span}\, \{ \psib_k \}_{k=1}^N$. Thus, both 
$\mbT$ and $\wT$ are matrices in $\mathbb{C}^{N \times N}$
and it holds that $\mbT V = V \wT^\top$, where $V=[\psib_1,\dots,\psib_N]$.
This means that $\mbT$ and $\wT$ are similar matrices and, thus, have the same spectrum.
Hence, using Theorem \ref{thm:eqiv1} we obtain a finite-dimensional counterpart of 
Theorem \ref{thm:conv:orth}, which does not require the orthogonality of $\{ \psib_k \}_{k=1}^N$.

\begin{theorem}[Convergence of a two-level method in finite-dimension]\label{thm:conv:non_orth}
Assume that $\mbX = \C^N$ and let the hypotheses (H1), (H2), (H3) and (H4) be satisfied.
If $V_c = {\rm span}\,\{ \psib_k \}_{k=1}^{m}$ (with $m=\wm<N$), then
$$
\rho(\mbT) = \rho(\wT) = | \lambda_{m+1} | <1.
$$
\end{theorem}

The coarse space $V_c = {\rm span} \, \{ \psib_k \}_{k=1}^{m}$
is not necessarily (asymptotically) optimal. A different choice can lead to
better asymptotic convergence or even to a divergent two-level method.
To show these results, we consider an analysis based on the perturbation 
of functions belonging to the coarse space $V_c = {\rm span} \, \{ \psib_k \}_{k=1}^{m}$.
We have seen in Theorem \ref{thm:eqiv1}, that an eigenvector of $\mbG$ is in the kernel of the two-level operator $\mbT$ if and only if it belongs to $V_c$. Assume that the coarse space cannot represent exactly one eigenvector $\psib$ of $\mbG$. How is the convergence of the method affected?
Let us perturb the coarse space $V_c$ using the eigenvector $\psib_{m+1}$, that is
$V_c(\varepsilon) := \Span \, \{ \psib_j + \varepsilon \, \psib_{m+1} \}_{j=1}^m$.
Clearly, $\text{dim}\, V_c(\varepsilon) = m$ for any $\varepsilon \in \mathbb{R}$.
In this case, \eqref{eq:Tn} holds with $\wm = m+1$ and 
$\wT \in \mathbb{C}^{N \times N}$ becomes 
\begin{equation}\label{Ciaramella_mini_10_thm:TTTT}
\wT(\varepsilon) = 
\begin{bmatrix}
\wT_{\wm}(\varepsilon) & 0 \\
X(\varepsilon) & \Lambda_{\wm}\\
\end{bmatrix},
\end{equation}
where we make explicit the dependence on $\varepsilon$.
Notice that $\varepsilon=0$ clearly leads to 
$\wT_{\wm}(0)=\text{ diag}\, (0,\dots,0,\lambda_{m+1}) \in \mathbb{C}^{\wm \times \wm}$,
and we are back to the unperturbed case with $\wT(0)=\wT$ having spectrum $\{0,\lambda_{m+1},\dots,\lambda_{N}\}$. Now, notice that
$\min_{\varepsilon \in \mathbb{R}} \rho(\wT(\varepsilon)) \leq \rho(\wT(0)) = | \lambda_{m+1} |$.
Thus, it is natural to ask the question: is this inequality strict?
Can one find an $\weps \neq 0$ such that
$\rho(\wT(\weps))=\min_{\varepsilon \in \mathbb{R}} \rho(\wT(\varepsilon))<\rho(\wT(0))$
holds? If the answer is positive, then we can conclude that choosing the coarse vectors
equal to the dominating eigenvectors of $\mbG$ is not an optimal choice.
Moreover, one could ask an opposite question: can one find a perturbation of the eigenvectors 
that leads to a divergent method ($\rho(\wT(\varepsilon))>1$)?
The next key result provides precise answers to these questions in the case $m=1$.

\begin{theorem}[Perturbation of $V_c$]\label{Ciaramella_mini_10_thm:perturb}
Let $(\psib_1,\lambda_1)$, $(\psib_2,\lambda_2)$ and $(\psib_3,\lambda_3)$ be three eigenpairs of $\mbG$, $\mbG \psib_j = \lambda_j \psib_j$ such that $0<|\lambda_3|<|\lambda_2| \leq |\lambda_1|$, $\| \psib_j \|_2 =1$, $j=1,2$, and denote with $\wla_j$ the eigenvalues of $A$ corresponding to $\psib_j$.
Assume that both $\lambda_j$ and $\wlambda_j$ are real for $j=1,2$ and
$\wla_1\wla_2>0$ \footnote{The hypothesis $\wla_1\wla_2>0$ is not restrictive. The same calculations can be performed for $\wla_1\wla_2<0$, as the sign of the product only influences the sign of the derivative $\frac{d \lambda(\varepsilon,0)}{d\varepsilon}.$} Define
$V_c := \Span\,\{ \psib_1 + \varepsilon \psib_2 \}$ with $\varepsilon \in \mathbb{R}$, and
$\gamma := \langle \psib_1 , \psib_2 \rangle \in [-1,1]$. Then
\begin{itemize}\itemsep0em
\item[{\rm (A)}] The spectral radius of $\wT(\varepsilon)$ is
$\rho(\wT(\varepsilon))=\max\{ |\lambda(\varepsilon,\gamma)| , | \lambda_3 | \}$, where
\begin{equation}\label{Ciaramella_mini_10_thm:lam}
\lambda(\varepsilon,\gamma) = \frac{\lambda_1 \wla_2 \varepsilon^2 + \gamma(\lambda_1 \wla_2 + \lambda_2 \wla_1)\varepsilon + \lambda_2 \wla_1}{\wla_2 \varepsilon^2 + \gamma (\wla_1+\wla_2)\varepsilon + \wla_1}.
\end{equation}

\item[{\rm (B)}] Let $\gamma=0$.
If $\lambda_1>\lambda_2>0$ or $0>\lambda_2>\lambda_1$, then
$\min\limits _{\varepsilon \in \mathbb{R}} \rho(\wT(\varepsilon)) = \rho(\wT(0))$.

\item[{\rm (C)}] Let $\gamma=0$,
If $\lambda_2>0>\lambda_1$ or $\lambda_1>0>\lambda_2$, then there exists an $\weps \neq 0$ such that
$\rho(\wT(\weps)) = |\lambda_3| = \min\limits_{\varepsilon \in \mathbb{R}} \rho(\wT(\varepsilon)) < \rho(\wT(0))$.

\item[{\rm (D)}] Let $\gamma\neq 0$.
If $\lambda_1>\lambda_2>0$ or $0>\lambda_2>\lambda_1$, then
there exists an $\weps \neq 0$ such that $|\lambda(\weps,\gamma)|<|\lambda_2|$
and hence
$\rho(\wT(\weps)) = \max\{|\lambda(\weps,\gamma)|,|\lambda_3|\} < \rho(\wT(0))$.

\item[{\rm (E)}] Let $\gamma\neq 0$.
If $\lambda_2>0>\lambda_1$ or $\lambda_1>0>\lambda_2$, then 
there exists an $\weps \neq 0$ such that
$\rho(\wT(\weps)) = |\lambda_3| = \min\limits _{\varepsilon \in \mathbb{R}} \rho(\wT(\varepsilon)) < \rho(\wT(0))$.

\item[{\rm (F)}] The map $\gamma \mapsto \lambda(\varepsilon,\gamma)$ has a vertical asymptote at $\gamma^*(\varepsilon)=-\frac{\varepsilon^2 \wlambda_2 + \wlambda_1}{\varepsilon(\wlambda_1+\wlambda_2)}$
for any $\varepsilon^2 \neq - \frac{(\lambda_2 \wlambda_1)(\wlambda_1+\wlambda_2)}{\lambda_1\wlambda_2^2+\wlambda_1^2 \lambda_2}$. 
Thus there exits a neighborhood $I(\gamma^*)$ such that $\forall \gamma \in I(\gamma^*)$, $\lambda(\varepsilon,\gamma)\notin (-1,1)$.

\end{itemize}
\end{theorem}

\begin{proof}
Since $m=1$, a direct calculation allows us to compute the matrix
$$\wT_{\wm}(\varepsilon)=\begin{bmatrix}
 \lambda_1 - \frac{\lambda_1\wla_1(1+\varepsilon \gamma)}{g} &  -\varepsilon \frac{\lambda_1\wla_1(1+\varepsilon \gamma)}{g} \\
- \frac{\lambda_2\wla_2(\varepsilon + \gamma)}{g} & \lambda_2 - \frac{(\varepsilon\lambda_2\wla_2)(\varepsilon + \gamma)}{g} \\
\end{bmatrix},$$ 
where $g=\wla_1 + \varepsilon \gamma[ \wla_1+\wla_2] + \varepsilon^2 \wla_2$.
The spectrum of this matrix is $\{0, \lambda(\varepsilon,\gamma)\}$,
with $\lambda(\varepsilon,\gamma)$ given in \eqref{Ciaramella_mini_10_thm:lam}.
Hence, point ${\rm (A)}$ follows recalling \eqref{Ciaramella_mini_10_thm:TTTT}.

To prove points ${\rm (B)}$, ${\rm (C)}$, ${\rm (D)}$ and ${\rm (E)}$ we use some properties
of the map $\varepsilon \mapsto \lambda(\varepsilon,\gamma)$. First, we notice that 
\begin{equation}\label{Ciaramella_mini_10_thm:prop}
\lambda(0,\gamma)=\lambda_2, \; \lim_{\varepsilon \rightarrow \pm \infty} \lambda(\varepsilon,\gamma) = \lambda_1,
\; \lambda(\varepsilon,\gamma)=\lambda(-\varepsilon,-\gamma).
\end{equation}
Second, the derivative of $\lambda(\varepsilon,\gamma)$ with respect to $\varepsilon$ is
\begin{equation}\label{Ciaramella_mini_10_thm:der}
\frac{d \lambda(\varepsilon,\gamma)}{d \varepsilon}
= \frac{(\lambda_1-\lambda_2)\wla_1\wla_2(\varepsilon^2+2\varepsilon/\gamma+1)\gamma}{(\wla_2 \varepsilon^2+\gamma(\wla_1+\wla_2)\varepsilon+\wla_1)^2}.
\end{equation}
Because of $\lambda(\varepsilon,\gamma)=\lambda(-\varepsilon,-\gamma)$
in \eqref{Ciaramella_mini_10_thm:prop}, we can assume without loss of generality that
$\gamma \geq 0$.

Let us now consider the case $\gamma=0$. In this case, the derivative \eqref{Ciaramella_mini_10_thm:der}
becomes $\frac{d \lambda(\varepsilon,0)}{d \varepsilon}
= \frac{(\lambda_1-\lambda_2)\wla_1\wla_2 2\varepsilon}{(\wla_2 \varepsilon^2+\wla_1^2)^2}$.
Moreover, since $\lambda(\varepsilon,0)=\lambda(-\varepsilon,0)$ we can assume that $\varepsilon \geq 0$.

Case ${\rm (B)}$.
If $\lambda_1>\lambda_2>0$, then $\frac{d \lambda(\varepsilon,0)}{d \varepsilon}>0$ for all $\varepsilon>0$.
Hence, $\varepsilon \mapsto \lambda(\varepsilon,0)$ is monotonically increasing, $\lambda(\varepsilon,0) \geq 0$ for all $\varepsilon>0$ and, thus, the minimum of $\varepsilon \mapsto |\lambda(\varepsilon,0)|$
is attained at $\varepsilon = 0$ with $|\lambda(0,0)|=|\lambda_2|>|\lambda_3|$, and the result follows.
Analogously, if $0>\lambda_2>\lambda_1$, then $\frac{d \lambda(\varepsilon,0)}{d \varepsilon}<0$ 
for all $\varepsilon>0$.
Hence, $\varepsilon \mapsto \lambda(\varepsilon,0)$ is monotonically decreasing, 
$\lambda(\varepsilon,0) < 0$ for all $\varepsilon>0$ and
the minimum of $\varepsilon \mapsto |\lambda(\varepsilon,0)|$
is attained at $\varepsilon = 0$.

Case ${\rm (C)}$.
If $\lambda_1>0>\lambda_2$, then $\frac{d \lambda(\varepsilon,0)}{d \varepsilon}>0$
for all $\varepsilon >0$. Hence, $\varepsilon \mapsto \lambda(\varepsilon,0)$ is monotonically 
increasing and such that $\lambda(0,0)=\lambda_2<0$ and $\lim_{\varepsilon \rightarrow \infty} \lambda(\varepsilon,0) = \lambda_1>0$. Thus, the continuity of the map $\varepsilon \mapsto \lambda(\varepsilon,0)$
guarantees the existence of an $\weps>0$ such that $\lambda(\weps,0)=0$.
Analogously, if $\lambda_2>0>\lambda_1$, then $\frac{d \lambda(\varepsilon,0)}{d \varepsilon}<0$
for all $\varepsilon>0$ and the result follows by the continuity of $\varepsilon \mapsto \lambda(\varepsilon,0)$.

Let us now consider the case $\gamma>0$. The sign of 
$\frac{d \lambda(\varepsilon,\gamma)}{d \varepsilon}$ is affected by the term
$f(\varepsilon):=\varepsilon^2+2\varepsilon/\gamma+1$, which appears at the numerator
of \eqref{Ciaramella_mini_10_thm:der}.
The function $f(\varepsilon)$ is strictly convex, attains its minimum
at $\varepsilon=-\frac{1}{\gamma}$, and is negative in $(\bareps_1,\bareps_2)$
and positive in $(-\infty,\bareps_1)\cup(\bareps_2,\infty)$, with $\bareps_1,\bareps_2=-\frac{1\mp \sqrt{1-\gamma^2}}{\gamma}$.

Case ${\rm (D)}$.
If $\lambda_1>\lambda_2>0$, then $\frac{d \lambda(\varepsilon,\gamma)}{d \varepsilon}>0$ for all 
$\varepsilon > \bareps_2$. Hence, $\frac{d \lambda(0,\gamma)}{d \varepsilon}>0$, which means that
there exists an $\weps<0$ such that $|\lambda(\weps,\gamma)|<|\lambda(0,\gamma)|=|\lambda_2|$.
The case $0>\lambda_2>\lambda_1$ follows analogously.

Case ${\rm (E)}$.
If $\lambda_1>0>\lambda_2$, then $\frac{d \lambda(\varepsilon,\gamma)}{d \varepsilon}>0$
for all $\varepsilon>0$. Hence, by the continuity of $\varepsilon \mapsto \lambda(\varepsilon,\gamma)$
(for $\varepsilon\geq 0$) there exists an $\weps>0$ such that $\lambda(\weps,\gamma)=0$.
The case $\lambda_2>0>\lambda_1$ follows analogously.

Case ${\rm (F)}$. It is sufficient to observe that the denominator of 
$\lambda(\varepsilon,\gamma)$ is equal to zero for $\gamma=\gamma^*$, while the numerator
is nonzero and finite.
Hence, $\lim_{\gamma \rightarrow \gamma^*}|\lambda(\varepsilon,\gamma)|=+\infty$. 
As the map $\gamma \mapsto \lambda(\varepsilon,\gamma)$ is continuous in $(-\infty,\gamma^*) \cup (\gamma^*,+\infty)$, 
the result follows.
\end{proof}

Theorem \ref{Ciaramella_mini_10_thm:perturb} and its proof say that, if the two eigenvalues
$\lambda_1$ and $\lambda_2$ have opposite signs (but they could be equal in modulus),
then it is always possible to find an $\varepsilon \neq 0$ such that the 
coarse space $V_c := \Span\{ \psib_1 + \varepsilon \psib_2 \}$
leads to a faster method than $V_c := \Span\{ \psib_1 \}$, even though both are one-dimensional subspaces. 
In addition, if $\lambda_3 \neq 0$ the former leads to
a two-level operator $T$ with a larger kernel than the one corresponding to the latter.
The situation is completely different if
$\lambda_1$ and $\lambda_2$ have the same sign. In this case, the orthogonality parameter
$\gamma$ is crucial. If $\psib_1$ and $\psib_2$ are orthogonal ($\gamma=0$), then 
one cannot improve $V_c:= \Span\{ \psib_1 \}$ by a simple perturbation
using $\psib_2$. However, if $\psib_1$ and $\psib_2$ are not orthogonal ($\gamma \neq 0$),
then one can still find an $\varepsilon \neq 0$ such that $\rho(\wT(\varepsilon)) < \rho(\wT(0))$. 

Notice that, if $|\lambda_3|=|\lambda_2|$, Theorem \ref{Ciaramella_mini_10_thm:perturb} shows that one cannot obtain a $\rho(T)$ smaller than $|\lambda_2|$ using a one-dimensional perturbation. However, if one optimizes the entire coarse space $V_c$ (keeping $m$ fixed),
then one can find coarse spaces leading to better contraction factor of the two-level
iteration, even though $|\lambda_3|=|\lambda_2|$.

Theorem \ref{Ciaramella_mini_10_thm:perturb} has another important meaning.
If the eigenvectors $\psib_j$ are not orthogonal and one defines the coarse space $V_c$ 
using approximations to $\psib_j$, then the two-level method is not necessarily convergent. 
Even though the one-level iteration characterized by $\mbG$ is convergent, a wrong choice of coarse 
functions can lead to a divergent iteration. This phenomenon is observed numerically in
Section \ref{sec:num_exp}. However, the analysis performed in Theorem \ref{Ciaramella_mini_10_thm:perturb}
suggests a remedy to this situation.

\begin{corollary}[Correction of perturbed coarse space functions]\label{cor:corr_Vc}
Let the hypotheses of Theorem \ref{Ciaramella_mini_10_thm:perturb} be satisfied.
For any $r \in \mathbb{N}$ it holds that
\[ \mbG^r V_c=\spa\left\{\psib_1 +\varepsilon_r \psib_2\right\},\]
with $\varepsilon_r = \frac{\lambda_2^r}{\lambda_1^r}\varepsilon$.
Moreover, if the coarse space $V_c$ is replaced by $\mbG^r V_c$ (hence $\varepsilon$ is replaced by $\varepsilon_r $), 
there exists an $\widehat{r} \in \mathbb{N}$ such that $\rho(\varepsilon_{\widehat{r}},\gamma)<1$ for any $\gamma \in [-1,1]$. 
\end{corollary}
\begin{proof}
By computing
\[ \mbG^r V_c=\mbG^r\spa\left\{\psib_1+\varepsilon \psib_2\right\}=\text{span}\left\{\psib_1 +\frac{\lambda_2^r}{\lambda_1^r}\varepsilon \psib_2\right\},\]
one obtains the first statement.
The second statement follows from Theorem \ref{Ciaramella_mini_10_thm:perturb}, which guarantees that $\rho(0,\gamma)=|\lambda_2|<1$.
Since the map $\varepsilon \mapsto \rho(\varepsilon,\gamma)$ is continuous 
and $|\lambda_2|/|\lambda_1|<1$, there exists
a sufficiently large $r\in \mathbb{N}$ such that $\rho(\varepsilon_r,\gamma)<1$ holds.
\end{proof}

Corollary \ref{cor:corr_Vc} has the following important consequence.
If some ``bad-convergent'' eigenvectors of $\mbG$ are not sufficiently well represented 
by the coarse space functions, one can apply $r$ smoothing steps to the coarse space functions. 
The new space $\mbG^r V_c$ is a better approximation to the ``bad-convergent'' eigenfunctions of $\mbG$. 
Therefore, one can replace $V_c$ by $\mbG^r V_c$ to improve
the convergence properties of the two-level method.

\subsection{Local coarse functions}\label{subs:localCF}

In this section, we consider an operator $\mbG$ having the block form
$$
\mbG = \begin{bmatrix}
0 & \mbG_1 \\
\mbG_2 & 0  \\
\end{bmatrix}
$$
and defined on the space $\mbX := \wmbX \times \wmbX$, where $\wmbX$ is a Hilbert space
endowed by an inner product $\llang \cdot , \cdot \rrang$.
The corresponding operator $\mbA$ is $\mbA = \mbI - \mbG$.
Moreover, we assume that the operators $\mbG_j$, $j=1,2$, have the same eigenvectors $\{\psi_k\}_{k=1}^\infty$
forming an orthonormal basis of $\wmbX$ with respect to $\llang \cdot , \cdot \rrang$.
The eigenvalues of $\mbG_j$, for $j=1,2$, are denoted by $\theta_j(k)$.
This is exactly the structure of the substructured domain decomposition problem introduced in 
Section \ref{sec:Two-Level-method-Laplace} and corresponding to two subdomains,
as the following examples show.

\begin{example}\label{ex:rect1}
Consider a rectangle $\Omega:=(-L_1,L_2)\times (0,\widetilde{L})$, $\widetilde{L},L_1,L_2>0$
that is decomposed as $\Omega=\Omega_1 \cup \Omega_2$ by two overlapping subdomains 
$\Omega_1:=(-L_1,\delta)\times (0,\widetilde{L})$ and 
$\Omega_2:=(-\delta,L_2)\times (0,\widetilde{L})$ for some $0<\delta<\min(L_1,L_2)$.
The two interfaces are $\Gamma_1:=\{\delta\}\times (0,\widetilde{L})$ and 
$\Gamma_2:=\{-\delta\}\times (0,\widetilde{L})$. 
If $\mathcal{L}=-\Delta$, then the Schwarz operators $\mbG_1$ and $\mbG_2$ are
diagonalized by the sine-Fourier functions 
$\psi_k(y)=\sin(k y \pi /\widetilde{L})$,
for $k=1,2,\dots$
The eigenvalues of $\mbG_j$ are 
$\theta_j(k)=\sinh\left(\frac{k\pi}{\widetilde{L}}(L_j-\delta)\right) / \sinh\left(\frac{k\pi}{\widetilde{L}}(L_j+\delta)\right)$,
for $j=1,2$; see, e.g., \cite{Gander2011,Gabriele-Martin}.
\end{example}

\begin{example}\label{ex:circles}
Consider a disc $\Omega$ of radius $r$ and centered in the origin. One can decompose
$\Omega$ as the union of two overlapping subdomains $\Omega_1$ and $\Omega_2$, where
$\Omega_1$ is a disc of radius $r_1 < r$ and centered in the origin, and $\Omega_2$
is an annulus of external radius equal to $r$ and internal radius $r_2 \in (r_1,r)$.
If $\mathcal{L}=-\Delta+\eta$ with $\eta >0$, then the two Schwarz operators $\mbG_1$ and
$\mbG_2$ are diagonalized by periodic Fourier functions defined on circles; see, e.g., \cite{GanderXu1}.
\end{example}


Now, we assume that $V_c:=(\spa \{\psi_1,\psi_2,\cdots,\psi_m\})^2\subset \mbX$.
Prolongation and restriction operators are given (as in \eqref{prolungation-restriction_intro}) by
\begin{equation}\label{prolungation-restriction}
\medmuskip=-0.3mu
\thinmuskip=-0.3mu
\thickmuskip=-0.3mu
\nulldelimiterspace=0.05pt
\scriptspace=0.05pt    
\arraycolsep0.05em
P\begin{bmatrix}
{\bf v}\\
{\bf w}\\
\end{bmatrix}
:=\begin{bmatrix}
\sum\limits_{j=1}^{m}  ({\bf v})_j\psi_j, &
\sum\limits_{j=1}^{m}  ({\bf w})_j\psi_j
\end{bmatrix}^\top, 
\quad
R\begin{bmatrix}
f\\g
\end{bmatrix}:=\begin{bmatrix}
\llang \psi_1,f\rrang, &
\cdots, &
\llang \psi_m,f\rrang, &
\llang \psi_1,g\rrang, &
\cdots, &
\llang \psi_m,g \rrang
\end{bmatrix}^\top.
\end{equation}
The restriction of $\mbA$ onto the coarse space $V_c$ is $\mbA_c=R\mbA P$.
Notice that, since in this case $\mbA (V_c) \subseteq V_c$,
Theorem \ref{lemma:invAc} guarantees that the operator $\mbA_c$ is invertible.
Now, we study the spectral properties of $\mbT$ defined in \eqref{eq:itT}.

\begin{theorem}[Convergence of the two-level method with local coarse space functions]\label{thm:S2S_conv}
Consider the coarse space $V_c=(\spa \{\psi_1,\psi_2,\cdots,\psi_m\})^2$ and the operators $P$ and
$R$ defined in \eqref{prolungation-restriction}. 
All pairs $(\psi_k,\psi_{\ell})$ with $k,\ell\leq m$ are in the kernel of the operator $\mbT$.
Moreover, for any $S \in \mathcal{L}(\mathcal{X})$ denote by  $\|S\|_{\op}:=\sup\limits_{\|{\bf v}\|_{\infty}=1}\|S{\bf v}\|_{\infty}$, where $\|{\bf v}\|_{\infty}:=\max_{j=1,2} \|v_j\|$, with $\|v_j\|^2=\llang v_j$, $v_j\rrang$.
If the eigenvalues $\theta_{j}(k)$, $j=1,2$, are in absolute value 
non-increasing functions of $k$, then the spectral radius of $\mbT$,
$\rho(\mbT):=~\lim\limits_{n\rightarrow \infty}\|\mbT^n\|_{\op}^{\frac{1}{n}}$,  is given by
\begin{equation*}
\medmuskip=0mu
\thinmuskip=0mu
\thickmuskip=0mu
\nulldelimiterspace=1.5pt
\scriptspace=1.5pt    
\arraycolsep1.5em
\rho(\mbT)=
\begin{cases}
|\theta_1(m+1)\theta_2(m+1)|^{\frac{n_1+n_2}{2}}, \text{ if } n_1, n_2 \text{ are both even or odd},\\
|\theta_1(m+1)\theta_2(m+1)|^{\frac{n_1+n_2-1}{2}}\max\{|\theta_1(m+1)|,|\theta_2(m+1)|\},
\text{ otherwise.}
\end{cases}
\end{equation*}
\end{theorem}

\begin{proof}
Let us suppose that both $n_1$ and $n_2$ are even. The other cases can be treated similarly.
For $n_1$ even we define $\pi^{n_1}(k):=\theta^{\frac{n_1}{2}}_1(k)\theta^{\frac{n_1}{2}}_2(k)$ 
and study the action of the operator $\mbT$ on a vector $\left[ \psi_k, \psi_{\ell} \right]^\top$:
\begin{equation*}
\mbT \begin{bmatrix}
\psi_k\\
\psi_{\ell}
\end{bmatrix}
=\mbG^{n_2}(\mbI-P\mbA_c^{-1}R\mbA )\mbG^{n_1}
\begin{bmatrix}
\psi_k\\
\psi_{\ell}
\end{bmatrix}.
\end{equation*}
We begin with the case $k\leq m$ and $\ell\leq m$. First, let us compute the action of the
operator $R\mbA \mbG^{n_1}$ on $\left[ \psi_k, \psi_{\ell} \right]^\top$.
Since the operators $\mbG_j$ are diagonalized by the basis $\{ \psi_k\}_k$ one obtains
$\mbG^{n_1}
\begin{small}
\begin{bmatrix}
\psi_k\\
\psi_{\ell}
\end{bmatrix}
= \begin{bmatrix}
\pi^{n_1}(k)\psi_k\\ \pi^{n_1}(\ell)\psi_{\ell}
\end{bmatrix}
\end{small}$.
The action of $\mbA$ on $\left[ \pi^{n_1}(k)\psi_k, \pi^{n_1}(\ell)\psi_{\ell} \right]^\top$ is
\begin{equation*}
A\begin{bmatrix}
\pi^{n_1}(k)\psi_k\\ \pi^{n_1}(\ell)\psi_{\ell}
\end{bmatrix}=\begin{bmatrix}
I_d & -\mbG_1\\
-\mbG_2 & I_d
\end{bmatrix}\begin{bmatrix}
\pi^{n_1}(k)\psi_k\\ \pi^{n_1}(\ell)\psi_{\ell}
\end{bmatrix}=\begin{bmatrix}
\pi^{n_1}(k)\psi_k\\ \pi^{n_1}(\ell)\psi_{\ell}
\end{bmatrix}-\begin{bmatrix}
\pi^{n_1}(\ell)\theta_1(\ell)\psi_{\ell}\\ 
\pi^{n_1}(k)\theta_2(k)\psi_k\end{bmatrix}.
\end{equation*}
Since $\mbA$ is invertible and has the form $\mbA=\mbI - \mbG$,
the eigenvalues $\theta_j(k)$ must different from one. Hence, the product
$\mbA \left[ \pi^{n_1}(k)\psi_k, \pi^{n_1}(\ell)\psi_{\ell} \right]^\top \neq 0$.
Now, the application of the restriction operator $R$ on 
$\mbA \left[ \pi^{n_1}(k)\psi_k, \pi^{n_1}(\ell)\psi_{\ell} \right]^\top$ gives us
\begin{equation*}
R\mbA \begin{bmatrix}
\pi^{n_1}(k)\psi_k\\ \pi^{n_1}(\ell)\psi_{\ell}
\end{bmatrix}= \begin{bmatrix}
\pi^{n_1}(k){\bf e}_{k}\\ \pi^{n_1}(\ell){\bf e}_{\ell}
\end{bmatrix}-\begin{bmatrix}
\pi^{n_1}(\ell)\theta_1(\ell){\bf e}_{\ell}\\ \pi_1^{n_1}(k)\theta_2(k){\bf e}_{k}\end{bmatrix}=\Lambda \begin{bmatrix}
\pi^{n_1}(k){\bf e}_{k}\\ \pi^{n_1}(\ell){\bf e}_{\ell}
\end{bmatrix},
\end{equation*}
where ${\bf e}_{k}$ and ${\bf e}_{\ell}$ are canonical vectors 
in $\mathbb{R}^m$ and
$\Lambda :=\begin{small}
\begin{bmatrix}
I & -\theta_1(\ell)I\\
-\theta_2(k)I & I
\end{bmatrix}
\end{small}$,
with $I$ the $m\times m$ identity matrix.
We have then obtained
\begin{equation}\label{actionRAG}
R\mbA \mbG^{n_1} \begin{bmatrix}
\psi_k\\ \psi_{\ell}
\end{bmatrix}=\Lambda  \begin{bmatrix}
\pi^{n_1}(k){\bf e}_{k}\\ \pi^{n_1}(\ell){\bf e}_{\ell}
\end{bmatrix}.
\end{equation}
Now, by computing
\begin{equation*}
\medmuskip=-0.1mu
\thinmuskip=-0.1mu
\thickmuskip=-0.1mu
\nulldelimiterspace=0.05pt
\scriptspace=0.05pt    
\arraycolsep0.05em
\begin{split}
\mbA_c\begin{bmatrix}
\pi^{n_1}(k){\bf e}_{k}\\ \pi^{n_1}(\ell){\bf e}_{\ell}
\end{bmatrix}
&=R\mbA
\begin{bmatrix}
\pi^{n_1}(k)\psi_k\\ \pi^{n_1}(\ell)\psi_{\ell}
\end{bmatrix}
=R \begin{bmatrix}
\pi^{n_1}(k)\psi_k-\pi^{n_1}(\ell)\theta_1(\ell)\psi_{\ell}\\ \pi^{n_1}(\ell)\psi_{\ell}-\pi^{n_1}(k)\theta_2(k)\psi_k
\end{bmatrix}=\Lambda \begin{bmatrix}
\pi^{n_1}(k){\bf e}_{k}\\ \pi^{n_1}(\ell){\bf e}_{\ell}
\end{bmatrix}
\end{split}
\end{equation*}
one obtains the action of $\mbA_c^{-1}$ on $\Lambda 
\begin{small}\begin{bmatrix}
\pi^{n_1}(k){\bf e}_{k}\\ \pi^{n_1}(\ell){\bf e}_{\ell}
\end{bmatrix}\end{small}$, that is
\begin{equation}\label{actionA_cinverse}
\begin{bmatrix}
\pi^{n_1}(k){\bf e}_{k}\\ \pi^{n_1}(\ell){\bf e}_{\ell}
\end{bmatrix}=\mbA_c^{-1} \Lambda \begin{bmatrix}
\pi^{n_1}(k){\bf e}_{k}\\ \pi^{n_1}(\ell){\bf e}_{\ell}
\end{bmatrix}.
\end{equation}
Using \eqref{actionRAG} and \eqref{actionA_cinverse} we have
\begin{equation}\label{actionT}
\begin{split}
(\mbI-&P \mbA_c^{-1}R\mbA ) \mbG^{n_1}\begin{bmatrix}
\psi_k \\ \psi_{\ell}
\end{bmatrix}
=\begin{bmatrix}
\pi^{n_1}(k)\psi_k\\ \pi^{n_1}(\ell)\psi_{\ell}
\end{bmatrix}-P \mbA_c^{-1}\Lambda \begin{bmatrix}
\pi^{n_1}(k){\bf e}_{k}\\ \pi^{n_1}(\ell){\bf e}_{\ell}
\end{bmatrix} \\
&=\begin{bmatrix}
\pi^{n_1}(k)\psi_k\\ \pi^{n_1}(\ell)\psi_{\ell}
\end{bmatrix}-P
\begin{bmatrix}
\pi^{n_1}(k){\bf e}_{k}\\ \pi^{n_1}(\ell){\bf e}_{\ell}
\end{bmatrix}
=\begin{bmatrix}
\pi^{n_1}(k)\psi_k\\ \pi^{n_1}(\ell)\psi_{\ell}
\end{bmatrix}-\begin{bmatrix}
\pi^{n_1}(k)\psi_k\\ \pi^{n_1}(\ell)\psi_{\ell}
\end{bmatrix}=0.
\end{split}
\end{equation}
This means that all the pairs $(\psi_k,\psi_{\ell})$
with $k\leq m$ and $\ell\leq m$ are in the kernel of $\mbT$. 
The result for $n_1$ odd follows by similar calculations.

Next, let us consider the case $k>m$ and $\ell\leq m$. 
Recalling that the basis $\{\psi_k\}_k$ is orthonormal, one has 
\begin{equation*}
R\mbA \mbG^{n_1}\begin{bmatrix}
\psi_k\\ \psi_{\ell}
\end{bmatrix}= R\left( \begin{bmatrix}
\pi^{n_1}(k)\psi_{k}\\ \pi^{n_1}(\ell)\psi_{\ell}
\end{bmatrix}-\begin{bmatrix}
\pi^{n_1}(\ell)\theta_1(\ell)\psi_{\ell}\\ \pi^{n_1}(k)\theta_2(k)\psi_{k}\end{bmatrix}\right)
=\begin{bmatrix}
0 & -\theta_1(\ell)I\\ 0 & I
\end{bmatrix} \begin{bmatrix}
0\\ \pi^{n_1}(\ell){\bf e}_{\ell}
\end{bmatrix}.
\end{equation*}
Similarly as before, we compute
\begin{equation*}
\mbA_c\begin{bmatrix}
0 \\ \pi^{n_1}(\ell){\bf e}_{\ell}
\end{bmatrix}
=R\mbA\begin{bmatrix}
0 \\ \pi^{n_1}(\ell)\psi_{\ell}
\end{bmatrix}
=R\begin{bmatrix}
-\pi^{n_1}(\ell)\theta_1(\ell)\psi_{\ell}\\\pi^{n_1}(\ell)\psi_{\ell}
\end{bmatrix}
=\begin{bmatrix}
0 & -\theta_1(\ell)I\\ 0 & I
\end{bmatrix}\begin{bmatrix}
0 \\ \pi^{n_1}(\ell){\bf e}_{\ell}
\end{bmatrix},
\end{equation*}
which implies that 
\begin{equation*}
\begin{bmatrix}
0 \\ \pi^{n_1}(\ell){\bf e}_{\ell}
\end{bmatrix}=\mbA_c^{-1}\begin{bmatrix}
0 & -\theta_1(\ell)I\\ 0 & I
\end{bmatrix}
\begin{bmatrix}
0 \\ \pi^{n_1}(\ell){\bf e}_{\ell}
\end{bmatrix}.
\end{equation*}
Thus, we have 
\begin{equation}\label{actionT2}
\begin{split}
\mbT \begin{bmatrix}
\psi_k \\ \psi_{\ell}
\end{bmatrix}
&=\mbG^{n_2}\left(\begin{bmatrix}
\pi^{n_1}(k)\psi_k\\ \pi^{n_1}(\ell)\psi_{\ell}
\end{bmatrix}-P\mbA_c^{-1}
\begin{bmatrix}
0 & -\theta_1(\ell)I\\ 
0 & I\\
\end{bmatrix}
\begin{bmatrix}
0\\ \pi^{n_1}(\ell){\bf e}_{\ell}
\end{bmatrix}\right)\\
&=\mbG^{n_2}\left(\begin{bmatrix}
\pi^{n_1}(k)\psi_k\\ \pi^{n_1}(\ell)\psi_{\ell}
\end{bmatrix}-P
\begin{bmatrix}
0 \\ \pi^{n_1}(\ell){\bf e}_{\ell}
\end{bmatrix}\right)
=\begin{bmatrix}
\pi^{n_1+n_2}(k)\psi_k \\ 0
\end{bmatrix}.
\end{split}
\end{equation}
For the remaining case $k>m$ and $\ell>m$, the same arguments as before imply that
\begin{align}\label{actionT3}
\mbT \begin{bmatrix}
\psi_k \\ \psi_{\ell}
\end{bmatrix}
=\mbG^{n_2}(\mathbb{I}-P \mbA_c^{-1}R \mbA)\mbG^{n_1}
\begin{bmatrix}
\psi_k \\ \psi_{\ell}
\end{bmatrix}
=\mbG^{n_2}\mbG^{n_1}\begin{bmatrix}
\psi_k \\ \psi_{\ell}
\end{bmatrix}
=\begin{bmatrix}
\pi^{n_1+n_2}(k)\psi_k \\ \pi^{n_1+n_2}(\ell)\psi_{\ell} 
\end{bmatrix}.
\end{align}

We can now study the norm of $\mbT$. 
To do so, we first use \eqref{actionT}, \eqref{actionT2} and \eqref{actionT3},
and that $\{ \psi_k,\psi_{\ell}\}_{k,\ell}$ is a basis of $\mbX$, to write
\begin{equation*}
\mbT {\bf v} = 
\mbT \begin{bmatrix}
\sum_{k=1}^{\infty} {\bf c}_k \psi_k \\
\sum_{\ell=1}^{\infty} {\bf d}_\ell \psi_\ell \\
\end{bmatrix}
=
\mbT \begin{bmatrix}
\sum_{k=m+1}^{\infty} \pi(k) {\bf c}_k \psi_k \\
\sum_{\ell=m+1}^{\infty} \pi(\ell) {\bf d}_\ell \psi_\ell \\
\end{bmatrix},
\end{equation*}
for any ${\bf v} \in \mbX$. Since $|\theta_1(k)|$ and $|\theta_2(k)|$ are non-increasing functions of $k$,
$|\pi(k)|$ is also a non-increasing function of $k$. Therefore, using that the basis 
$\{ \psi_k,\psi_{\ell}\}_{k,\ell}$ is orthonormal, we get
\begin{equation*}
\|\mbT\|_{\op}
=\sup\limits_{\|{\bf v}\|_{\infty}=1}\|\mbT{\bf v}\|_{\infty}
\leq \max\left(|\pi^{n_1+n_2}(k)|,|\pi^{n_1+n_2}(\ell)|\right)=|\pi^{n_1+n_2}(m+1)|.
\end{equation*}
This upper bound is achieved at ${\bf v}=[\psi_{m+1} , 0]^\top$.
Hence, $\|\mbT\|_{\op} = |\pi^{n_1+n_2}(m+1)|$.
Now, a similar direct calculation leads to $\|\mbT^n\|_{\op} = |\pi^{n(n_1+n_2)}(m+1)|$,
which implies that
$\rho(\mbT)=\lim\limits_{n\rightarrow \infty} (\|\mbT^n\|_{\op})^{1/n}=|\pi^{n_1+n_2}(m+1)|$.
\end{proof}

Theorems \ref{thm:eqiv1}, \ref{thm:conv:non_orth} and \ref{thm:S2S_conv} show that the choice of the basis functions to construct $V_c$ can affect drastically the convergence of the method.
On the one hand, an inappropriate choice of $V_c$ can lead to a two-level method that performs as the corresponding one-level method.
On the other hand, a good choice of $V_c$ can even make convergent a non-converging stationary method;
see, e.g., \cite{CiaramellaRefl}.

\section{Numerical construction of the coarse space}\label{sec:numerical_Vc}

The construction of a good coarse space $V_c$ is not an easy task.
Several works rely on the solution of generalized eigenvalue problems on the interfaces;
see, e.g., \cite{Bjorstad2018,gander2017shem,gander2019song,Klawonn}; see also \cite{Chaouqui2019,CiaramellaRefl}. Despite one could re-use these techniques to build a coarse space for the S2S method, see the S2S-HEM method discussed in Section \ref{sec:num_exp}, we now present two alternative numerical approaches for the generation of coarse space functions.
The first one relies on the principal component
analysis (PCA) and share some similarities with some of the strategies
presented in \cite{Brezina2005,xu_zikatanov_2017}.
The second approach is based on modeling the two-level iteration
operator as a deep neural network where the coarse space functions are regarded 
as variables to be optimized. A similar approach has been presented
in the context of multigrid methods in \cite{katrutsa2017deep}.

We remark that the S2S framework facilitates the use of these two numerical techniques which could be even numerically unfeasible if applied to a two-level volume method. Indeed, at the discrete level, the substructured coarse functions are much shorter vectors than the corresponding volume ones. This means that, for the PCA approach, one has to compute the SVD decomposition of a much smaller matrix, while for the deep neural network approach, the neural net has much less parameters to optimize.

\subsection{A PCA approach for an automatic coarse space generation}\label{sec:random_S2S}

The idea that we present in this section
is to construct an approximation of the image of the smoother $G^r$, for some positive integer $r$.
In fact, the image of $G^r$ contains information about the ``bad converging'' eigenvectors of $G$.
Notice that ${\rm im}(G^r)={\rm im}(G^rX)$ for any surjective matrix $X$. Therefore, the idea is
to construct a coarse space using the information contained in $G^rX$, for some randomly chosen matrix $X$.
Clearly, if $\rho(G)<1$ and $r$ is large, then one expects that the slowest convergent
eigenvectors are predominant in $G^rX$. Notice also the relation of this idea with the perturbation
Theorem \ref{Ciaramella_mini_10_thm:perturb} and Corollary \ref{cor:corr_Vc}.

Motivated by these observations, we use a principal component analysis (PCA),
also known as proper orthogonal decomposition (POD); see, e.g., \cite{Gubisch2017} and references therein.
We consider the following procedure.

\begin{enumerate}\itemsep0em
\item Consider a set of $q$ linearly independent randomly generated vectors \linebreak
$\{ {\bf s}_k \}_{k=1}^q \subset \mathbb{R}^{N^s}$, where $N^s$ is the number of degrees of freedom on the interfaces, and define the matrix
$S = [ {\bf s}_1 ,\cdots, {\bf s}_q ]$. Here, $q \approx m$ and $m$ is the desired dimension of the coarse space.
\item Use the vectors ${\bf s}_k$ as initial vectors and perform $r$ smoothing steps to create the
matrix $W = G^r S$.
This computation can be performed in parallel and we assume that $r$ is ``small''.
\item Compute the SVD of $W$: $W=U \Sigma V^\top$. This is cheap ($O(q(N^s)^2)$) because 
$W\in \mathbb{R}^{N^s \times q}$ is ``small'', since $q$ is ``small'' and ${\bf v}_k$ are interface vectors.
\item Since the left-singular vectors (corresponding to the non-zero singular values) 
span the image of $W$, we define $V_c:= \spa\{ {\bf u}_j \}_{j=1}^{m}$ and $P:=[{\bf u}_1,\cdots,{\bf u}_{m}]$.
\end{enumerate}

We wish to remark that, in light of Theorem \ref{Ciaramella_mini_10_thm:perturb} and Corollary \ref{cor:corr_Vc},
one can also use approximations of the eigenfunctions of $G$ (if available) in the matrix $S$ (in step 1 above).
A numerical study of the above procedure is given in Section \ref{sec:num_exp}.
To qualitatively describe the obtained coarse space, we prove the following bound.

\begin{lemma}[Approximation of the random generated coarse space]\label{lemma:randomPCS}
Consider a full rank orthogonal matrix $X \in \mathbb{R}^{N^s \times N^s}$ and its decomposition $X=[S,\widetilde{S}]$.
Let $W=\mbG^r [S,0]$ and $P_\ell = U_\ell \Sigma_\ell V_\ell^\top$ be the rank-$\ell$ SVD of $W$ ($\ell \leq m$),
where $(\Sigma_\ell)_{j,j} = \sigma_j$, $j=1,\dots,\ell$ are the singular values of $W$.
Then, it holds that
$$\| P_\ell - G^r X \|_2 \leq \sigma_{\ell+1} + \| G^r \|_2.$$
\end{lemma}
\begin{proof}
Using the triangle inequality, we get
$$\| P_\ell - G^r(X) \|_2 \leq \| P_\ell - G^r [S,0] \|_2 + \| G^r [S,0]  - G^r X \|_2.$$
The first term on the right-hand side is equal to $\sigma_{\ell+1}$ by the best approximation
properties of the SVD. The second term can be bounded as $\| G^r [S,0] - \mbG^r X \|_2\leq \| G^r \|_2 \| [S,0] - X \|_2$,
and a direct calculation of $\| [S,0] - X \|_2=\| [0,\widetilde{S}] \|_2$ leads to the result as $\widetilde{S}^\top\widetilde{S}=I_{N_s-q}$.
\end{proof}

Despite its very simple proof, Lemma \ref{lemma:randomPCS} allows us to describe the
quality of the created coarse space.
Larger values of $q$ and $\ell$ lead to a smaller error in 
the approximation of the image of $G$. Moreover, a smoother $G$ with good contraction properties,
namely $\|G \|_2 \ll 1$,
leads to a better approximation. Clearly, one can improve the approximation by enlarging $r$ at the cost
of extra subdomain solves. 

\subsection{Generating the coarse space by deep neural networks}\label{sec:deep_neural}
Theorem \ref{Ciaramella_mini_10_thm:perturb} shows that the spectral coarse space made by the first slowest eigenvector of $\mbG$ is not necessarily the one-dimensional coarse space minimizing $\rho(\mbT)$. 
Now, we wish to go beyond this one-dimensional analysis and optimize
the entire coarse space $V_c$ keeping its dimension $m$ fixed. This is equivalent to
optimize the prolongation operator $P$ whose columns span $V_c$. 
Thus, we consider the optimization problem
\begin{equation}\label{Ciaramella_mini_10_eq:optimization_problem}
\min_{P \in\mathbb{R}^{N^s\times m}} \rho(T(P)).
\end{equation}
To solve approximately \eqref{Ciaramella_mini_10_eq:optimization_problem}, we follow 
the approach proposed by \cite{katrutsa2017deep}.
Due to the Gelfand formula $\rho(T)=\lim_{k\rightarrow \infty} \sqrt[k]{\|T^k\|_F}$, we replace \eqref{Ciaramella_mini_10_eq:optimization_problem} with the simpler optimization problem $\min_{P} \|T(P)^k\|^2_F$ for some positive $k$. Here, $\|\cdot\|_F$ is the Frobenius norm. We then consider the unbiased stochastic estimator \cite{hutchinson1989stochastic}
\[\|T^k\|^2_F=\text{trace}\left((T^k)^\top T^k\right)=\mathbb{E}_{\mathbf{z}}\left[ \mathbf{z}^\top (T^k)^\top T^k  \mathbf{z}\right]=\mathbb{E}_{\mathbf{z}}\left[ \|T^k \mathbf{z}\|^2_2\right] ,\]
where $\mathbf{z}\in\mathbb{R}^{N^s}$ is a random vector with Rademacher distribution, i.e. $\mathbb{P}(\mathbf{z}_i=\pm 1)=1/2$. Finally, we rely on a sample average approach, replacing the unbiased stochastic estimator with its empirical mean such that \eqref{Ciaramella_mini_10_eq:optimization_problem} is approximated by
\begin{equation}\label{Ciaramella_mini_10_eq:optimization_problem_emp}
\min_{P \in\mathbb{R}^{N^s\times m}} \frac{1}{N}\sum_{i=1}^N\|T(P)^k \mathbf{z}_i\|^2_F, 
\end{equation}
where $\mathbf{z}_i$ are a set of independent, Rademacher distributed, random vectors.
The action of $T$ onto the vectors $\mathbf{z}_i$ can be interpreted as the feed-forward process of a neural net, where each layer represents one specific step of the two-level method, that is the smoothing step, the residual computation, the coarse correction and the prolongation/restriction operations. 
In our setting, the weights of most layers are fixed and given, and the optimization is performed only on the weights of the layer representing the prolongation step. The restriction layer is constraint to have as weights the transpose of the weights of the prolongation layer.
To solve \eqref{Ciaramella_mini_10_eq:optimization_problem_emp}, we rely on the stochastic gradient algorithm which requires at each iteration to compute $k$ times the action of $T$. This step is expensive as it is equivalent to perform $k$ iterations of the two-level method. Hence, the deep neural network approach is not computationally efficient to build coarse spaces, unless one considers an offline-online paradigm or in a many query context. We will use this approach in Section \ref{sec:num_exp} to go beyond the result of Theorem \ref{Ciaramella_mini_10_thm:perturb} and show numerically that given an integer $m$, a spectral coarse made by the first $m$ slowest eigenvectors of $G$ is not necessarily the asymptotic optimal coarse space of dimension $m$.

\section{Numerical experiments}\label{sec:num_exp}
This section is concerned with the numerical validation of the framework proposed in this manuscript.
We first consider a Poisson equation in 2D and 3D rectangular boxes and we show the convergence behavior of the S2S method with different coarse spaces and of the SHEM method (see, \cite{gander2017shem}).  In this simplified setting, we also report the computational time and memory storage requirements of the S2S and SHEM methods. We then solve a Poisson problem with many-subdomain decompositions and discuss a further way to build a substructured coarse space, that is, using the SHEM interface functions. Next, we focus on a diffusion problem with highly jumping coefficients and validate Theorem \ref{Ciaramella_mini_10_thm:perturb} showing how a perturbed coarse space can affect the convergence of the methods.

\subsection{Poisson equation in 2D and 3D rectangular boxes}
Let us consider a rectangular domain $\Omega=\Omega_1\cup \Omega_2$, where $\Omega_1=(-1,\delta)\times(0,1)$ and $\Omega_2=(-\delta,1)\times(0,1)$, and a Poisson equation $-\Delta u=f$ with homogeneous Dirichlet boundary condition on $\partial \Omega$. 
Given an integer $\ell \geq 2$, we discretize each subdomain with a standard second order finite difference scheme with $N_y=2^\ell -1$ points in the $y$ direction and $N_x=N_y$ points in the $x$ direction. The overlap has size $2\delta$ with $\delta=N_{ov}h$, where $h$ is the mesh size and $N_{ov}\in\mathbb{N}$. In our computations, we consider $f=1$ and initialize the iterations with a random initial guess.

Figure \ref{Fig:Poisson} shows the relative error decay for several methods.
Specifically, we compare the one-level parallel Schwarz method ($G_s$ (Schwarz) in the figures), a S2S method with a coarse space made by eigenfunctions of $G$ (S2S-$G$), a S2S method
with a coarse space made of eigenfunctions of the operators $G_j$ (S2S-$G_j$), a S2S method with a coarse space
obtained with the PCA procedure (S2S-PCA), a S2S method with coarse space obtained using deep neural networks (S2S-DNN), and the spectral volume method based on the SHEM coarse space (SHEM), see \cite{gander2015analysis}.
For the PCA coarse space, we average the relative error decay over 30 realizations and the parameters for the PCA procedure are $q=2 \, \text{dim}V_c$ and $r=2$, where $\text{dim}V_c$ is the desired size of the spectral coarse space. For the deep neural network approach, the parameters are $N=N^s$ and $k=4$.

Figure \ref{Fig:Poisson} shows that most the spectral methods have a very similar convergence. 
Indeed, we have numerically observed that the S2S-$G$, the S2S-$G_j$ and the SHEM methods have all the same spectral radius in this simplified setting.
We remark that the S2S-PCA method has on average the same convergence behavior of the other two-level methods, even tough sometimes it could be slightly 
different (faster or slower). The S2S-DNN method outperforms the others. In this particular setting, the eigenvalues of $G$ are $\lambda_j=\pm \sqrt{\mu_j}$, where $\mu_j>0,\; \forall j=1,\dots,N^s$ are the eigenvalues of $G_1=G_2$, and $A$ is symmetric. Hence, we are under the assumptions of point $(C)$ of Theorem \ref{Ciaramella_mini_10_thm:perturb}, and Figure \ref{Fig:Poisson} confirms that a spectral coarse space is not necessarily the coarse space leading to the fastest convergence.

As we claimed that the deep neural network approach is computationally expensive, it is worth remarking that the PCA approach builds a coarse space as efficient as the spectral ones performing $q\cdot r$ subdomains solves in parallel, instead of solving eigenvalue problems as required by all others two-level methods, either locally (as the S2S-$G_j$ and SHEM methods) or on the whole skeleton (as the S2S-$G$ method).

\begin{figure}
\includegraphics[scale=0.40]{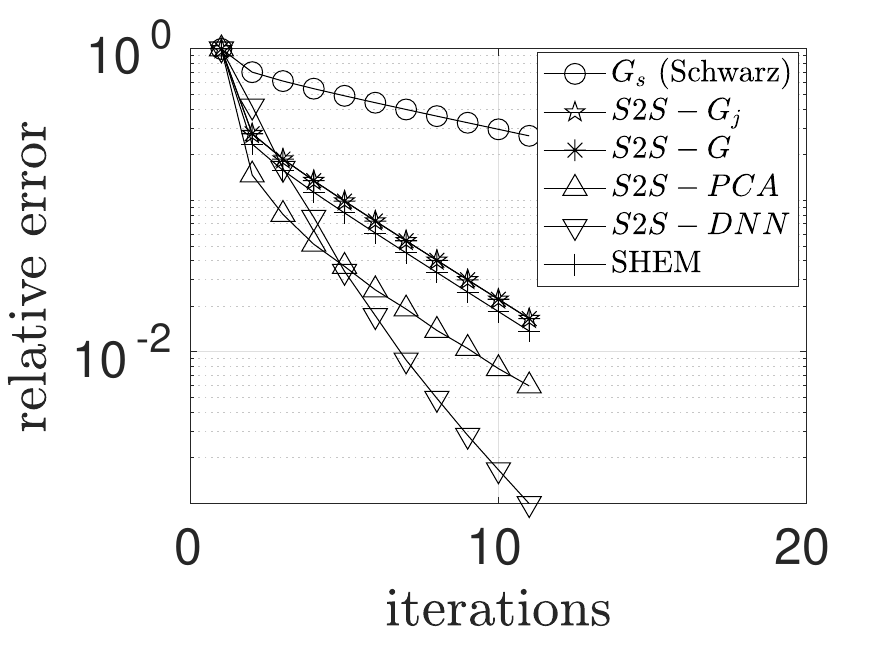}
\includegraphics[scale=0.40]{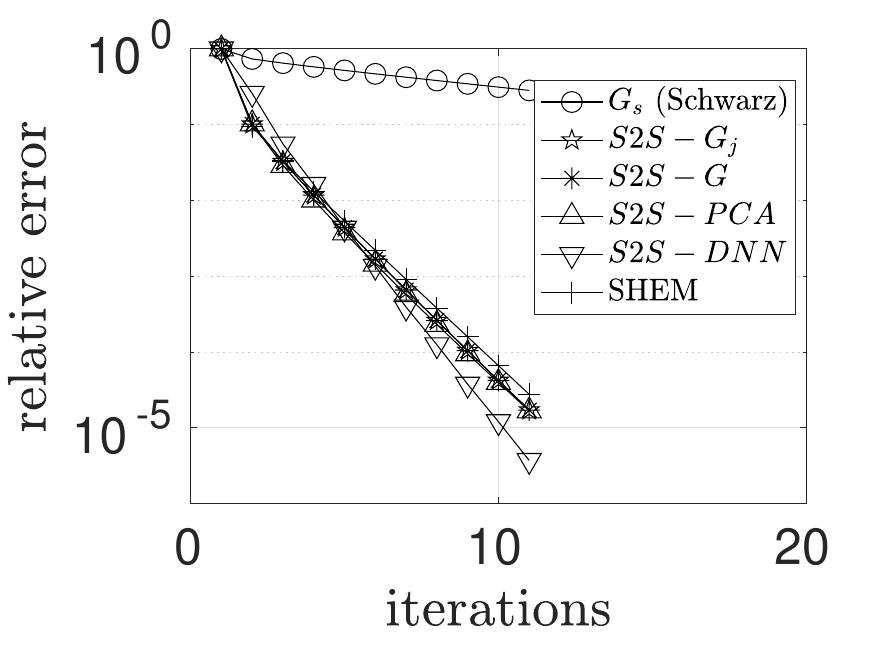}
\includegraphics[scale=0.40]{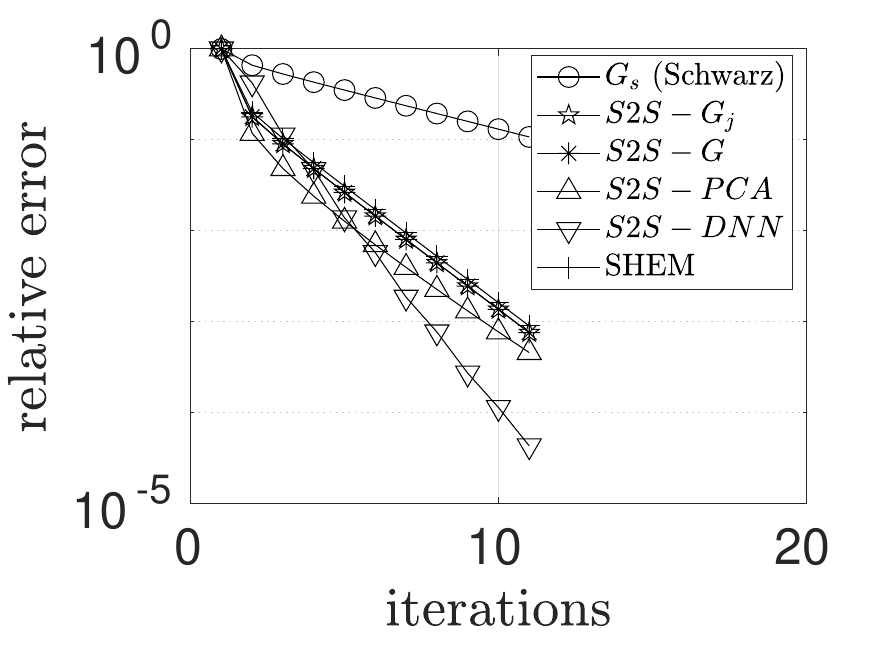}
\includegraphics[scale=0.40]{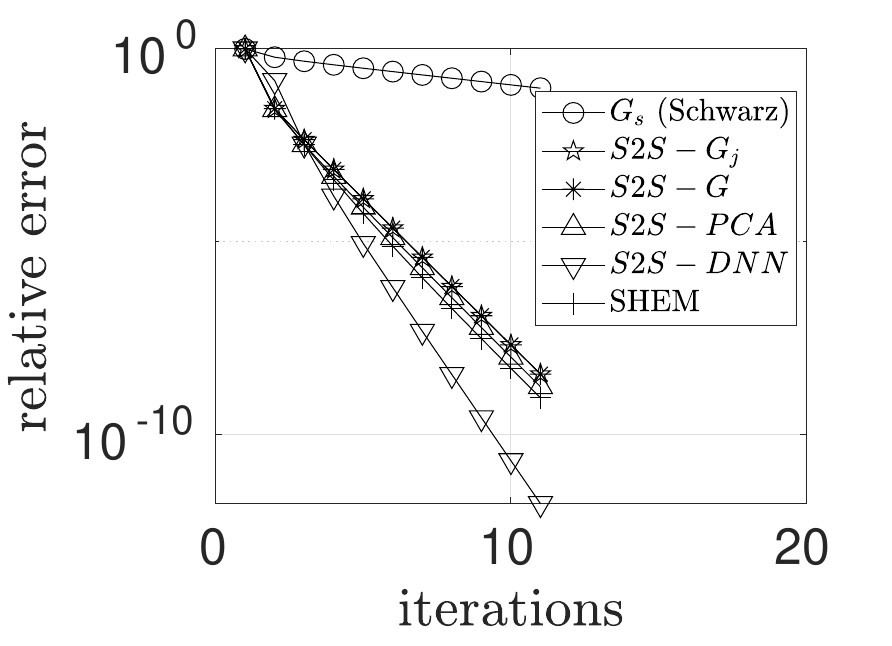}
\caption{Convergence curves for $\ell = 6$ which corresponds to 7875 unknowns. On the top row, $N_{ov} = 2$ while on the bottom row $N_{ov} = 4$. In the left panels $\text{dim}V_c=4$, while in the right panels $\text{dim}V_c=16$.}\label{Fig:Poisson}
\end{figure}

Next, we compare the computational costs required by the S2S method and a spectral volume method in Table \ref{table_cost}.
For simplicity we assume $n_1=1$, $n_2=0$. 
Let $A_v=M-N$ be a volume matrix of size $N^v\times N^v$ and $A$ be the substructured matrix of size $N^s \times N^s$, $P$ and $R$ the substructured restriction and prolongation operators, while $P_v$ and $R_v$ are the corresponding volume operators.
On each subdomain, we suppose to have $N_{\text{sub}}$ unknowns and $m$ is the dimension of the coarse space. The cost of the smoothing step is equal in both case to $\gamma_s(N_{\text{sub}})$, where $\gamma_s$ depends on the choice of the linear solver, e.g. for a Poisson problem, ${\gamma_s}(N_{\rm sub})=N_{\rm sub}\log(N_{\rm sub})$ if a
fast Poisson solver is used, 
or ${\gamma_s}(N_{\rm sub})=b N_{\rm sub}$
for sparse banded matrices with bandwidth $b$; see, e.g., \cite{govloan2013}.
Further, the cost of solving the coarse problem is identical as well, equal to  
$\gamma_c(m)$, where $m$ is the size of the coarse space and $\gamma_c$ depends on the linear solver used. The coarse matrices are usually small, fully dense, matrices so that it is reasonable to factorize them using an LU decomposition.
In a standard implementation, the S2S method requires to perform subdomain solves when computing the residual, as the matrix vector multiplication with the matrix $A$ is needed. To avoid this extra-cost, in the Appendix \ref{sec:appendix} we show two alternative algorithms to implement smartly the S2S method, where the residual is computed cheaply and the two applications of the smoothing operators per iteration are avoided. We further show that these two alternatives have the same convergence behavior of Algorithm \ref{two-level}.

The main advantage of the S2S method is that the restriction and prolongation operators are performed on the substructures, with objects which are smaller than the corresponding volume counterparts. Thus the S2S method naturally requires less memory storage. For instance, given a coarse space of dimension $m$, the restriction and prolongation operators are matrices of size $m\times N^s$ and $N^s\times m$ where $N^s$ is the number of unknowns on the substructures. The corresponding volume objects have size $m\times N^v$ and $N^v\times m$, where $N^v$ is the number of unknowns on the whole domain. Thus the S2S method presents advantages both from memory storage and from the computational time point of view, due to smaller number of floating point operations.


\begin{table}[]
\centering
\begin{small}
\begin{tabular}{l|l||l|l}
G2S &  G2S C.C. &  Volume two-level & Volume C.C.  \\ \hline
${\bf{v}}^{n+\frac{1}{2}}= G_h {\bf{v}}^{n} +{\bf b}_h$ & $O({\gamma_s}(N_{\text{sub}}))$ & ${\bf{u}}_v^{n+\frac{1}{2}}= N {\bf{u}}_v^{n} +M^{-1}{\bf b}_v$ & $O({\gamma_s}(N_{\text{sub}})$ \\
${\bf{r}}^{n+\frac{1}{2}}={\bf b}_h -A_h {\bf{v}}^{n+\frac{1}{2}}$&  $O({\gamma_c}(N_{\text{sub}}))$ & ${\bf{r}}_v^{n+\frac{1}{2}}={\bf b}_v -A_v {\bf{u}}_v^{n+\frac{1}{2}}$ & $O((N^v)^{\gamma_m})$ \\
${\bf{v}}_{c}^{n+1}=A_{2h}^{-1}(R{\bf{r}}^{n+\frac{1}{2}})$ & $O(\gamma_c(m))$ & ${\bf{u}}_{vc}^{n+1}=A_{vc}^{-1}(R_v{\bf{r}}_v^{n+\frac{1}{2}})$&  $O(\gamma_c(m))$\\
${\bf{v}}^{n+1}={\bf{v}}^{n+\frac{1}{2}} +P{\bf{v}}_{c}^{n+1}$ & $O(N^s)$ & ${\bf{u}}_v^{n+1}={\bf{u}}_v^{n+\frac{1}{2}} +P_v {\bf{u}}_{vc}^{n+1}$ & $O(N^v)$\\ \hline
\end{tabular}
\end{small}
\caption{Computational cost (C.C.) per iteration. Notice that the smoother in volume is written
as a standard stationary method based on the splitting $A_v=M-N$.}\label{table_cost}
\end{table}

We now discuss the cost of the off-line computation phases.
To build prolongation and restriction operators in the volume case, one needs to define some functions, usually by solving eigenvalue problems, along the interfaces between non-overlapping subdomains or in the overlap region between overlapping subdomains.
These functions are then extended in the interior of the subdomains and this extension costs $\gamma_s(N_{\text{sub}})$. Notice that the way of extending these functions is not unique way and we refer to \cite[Section 5]{Klawonn} for an overview. 
In the substructured framework, we have analyzed theoretically several ways among which a global eigenvalue problem (S2S-$G$), local eigenvalue problems (S2S-$G_j$), and randomized approaches using either PCA (S2S-PCA), or deep neural networks (S2S-DNN). The relative costs of these approaches with respect to the volume ones are difficult to estimate as they depend on the features of the problem at hand. 
Nevertheless, for any method used to generate the interface functions, we do not need to perform any extension step in the substructured framework.
Besides the approaches studied theoretically, we emphasize that one can use the interface functions computed in a volume method as a basis for the S2S coarse space. In this way one avoids the extension step and exploits at best the intrinsic substructured nature of the S2S method. In the next section we show numerical results where we used the SHEM interface functions as a basis for the S2S method (called the S2S-HEM method).
Finally, we observe the S2S-PCA approach is cheaper than any volume method if one sets, e.g., $r=1$, as the cost would be $\gamma_s(N_{\text{sub}})$ and thus cheaper then all volumes approaches since they require the additional cost of solving localized eigenvalues problems.  

To conclude, we consider a Poisson equation on a three-dimensional box $\Omega = (-1, 1) \times  (0, 1) \times 
(0, 1)$ decomposed into two overlapping subdomains $\Omega_1 = (-1, \delta) \times (0, 1) \times (0, 1)$ and $\Omega_2 =
(-\delta, 1) \times  (0, 1) \times  (0, 1)$. 
Table \ref{Tab:3D} shows the computational times to reach a relative error smaller than $10^{-8}$, and the computational memory required 
to store the restriction and interpolation operators in a sparse format in Matalb for the S2S method and the SHEM method. 
The experiments have been performed on a workstation with 8 processors Intel Core i7-6700 CPU 3.40GHz and with
32 GB of RAM.
\begin{table}[]
\centering
\begin{small}
\setlength{\tabcolsep}{5pt}
\begin{tabular}{ c | c c c }
$N_v$-\text{dim} $V_c$ &  6075-4 & 56699-16 &   488187-64  \\ \hline
S2S-$G$ & 0.1175 &  3.09  &  157.62  \\
SHEM & 0.1065  &  3.16 & 158.34 \\
\end{tabular}\quad
\begin{tabular}{ c | c c c }
$\text{dim} V_c$ &  6075-4 & 56699-16 & 488187-64  \\ \hline
S2S-$G$ & 0.0288  &  0.49 &  16.24  \\
SHEM & 0.77  & 14.51  & 749.84  \\
\end{tabular}
\end{small}
\caption{On the left, time in seconds required by the S2S-$G$ and SHEM methods to reach a relative error smaller than $10^{-8}$ for increasing number of unknowns $N_v$ and dimension of coarse space $V_c$. The overlap parameter is constant $N_{ov}=4$. On the right, memory usage expressed in megabyte to store the restriction and prolongation operators in the S2S and SHEM methods.}
\label{Tab:3D}
\end{table}
We remark that the S2S method requires drastically less memory than the SHEM method, which becomes inefficient for large problems from the memory point of view.
Concerning computational times, we observe that the two methods are equivalent in this setting. The substructured restriction and prolongation operators are faster than the volume ones, since to compute the action for instance of the substructured prolongation operator on the largest problem takes about $7\cdot 10^{-4}$ seconds compared to $3\cdot 10^{-2}$ seconds of the volume prolongation. However, the bottleneck here is represented by the two, very large, subdomain solves. A many subdomain decomposition and a parallel implementation on a high performance programming language should make more evident the advantage of using substructured coarse spaces in terms of computational time.

\subsection{Decomposition into many subdomains}
In this section, we consider a Poisson equation in a square domain $\Omega$ decomposed into $M\times M$ nonoverlapping square subdomains 
$\widetilde{\Omega}_j$, $j=1,...,M^2=N$. Each subdomain $\widetilde{\Omega}_j$ contains $N_{\text{sub}}:=(2^\ell-1)^2$ interior degrees of freedom.
The subdomains $\widetilde{\Omega}_j$ are extended by $N_{ov}$ points to obtain subdomains $\Omega_j$ which form an overlapping decomposition of $\Omega$. Each discrete local substructure is made by one-dimensional segments. Figure \ref{fig:subdomain} provides a graphical representation.

\begin{figure}
\centering
\begin{tikzpicture}[scale=0.15]
\draw[color=black] (10,20.5) node {$\Omega$};
\draw [thick,black] (1,1) rectangle (19,19);
\draw [thick,black] (1,7) rectangle (19,7);
\draw [thick,black] (1,13) rectangle (19,13);
\draw [thick,black] (7,1) rectangle (7,19);
\draw [thick,black] (13,1) rectangle (13,19);
\foreach \x in {2,3,4,5,6,7,8,9,10,11,12,13,14,15,16,17,18}
    \foreach \y in {2,3,4,5,6,7,8,9,10,11,12,13,14,15,16,17,18}
      {
        \draw (\x,\y) circle (0.2cm) ;
      }
\end{tikzpicture}\qquad
\begin{tikzpicture}[scale=0.15]
\draw[color=black] (10,20.5) node {$\Omega$};
\draw [thick,black] (1,1) rectangle (19,19);
\draw [thick,black] (1,7) rectangle (19,7);
\draw [thick,black] (1,13) rectangle (19,13);
\draw [thick,black] (7,1) rectangle (7,19);
\draw [thick,black] (13,1) rectangle (13,19);
\foreach \x in {2,3,4,5,6,7,8,9,10,11,12,13,14,15,16,17,18}
    \foreach \y in {2,3,4,5,6,7,8,9,10,11,12,13,14,15,16,17,18}
      {
        \draw (\x,\y) circle (0.2cm) ;
      }
\draw [thick,yellow] (1,1) rectangle (8,8);  
\draw [thick,violet] (6,6) rectangle (14,14); 
\draw [thick,green] (12,12) rectangle (19,19);      
\end{tikzpicture}\qquad 
\begin{tikzpicture}[scale=0.22]
\draw[color=black] (4,11) node {$\Omega_j$};
\draw [dashed,thick,black] (0,0) rectangle (8,8);
\draw [thick,violet] (-2,-2) rectangle (10,10);
\foreach \x in {1,2,3,4,5,6,7}
    \foreach \y in {1,2,3,4,5,6,7}
      {
        \draw (\x,\y) circle (0.2cm) ;
      }
\foreach \x in {-1,0,8,9}
    \foreach \y in {-1,0,1,2,3,4,5,6,7,8,9}
      {
        \draw [black] (\x,\y) circle (0.2cm) ;
      }
\foreach \y in {-1,0,8,9}
    \foreach \x in {-1,0,1,2,3,4,5,6,7,8,9}
      {
        \draw [black] (\x,\y) circle (0.2cm) ;
      }
  \draw [thick,blue] (-2,6) -- (10,6); 
  \draw [thick,blue] (2,-2) -- (2,10);
  \draw [thick,blue] (-2,2) -- (10,2);
  \draw [thick,blue] (6,-2) -- (6,10);    
\end{tikzpicture}
\caption{The domain $\Omega$ is divided into nine non-overlapping subdomains (left). The center panel shows how the diagonal non-overlapping subdomains are enlarged to form overlapping subdomains. On the right, we zoom on the central subdomain to show the local discrete substructure formed by the degrees of freedom lying on the blue segments.}
\label{fig:subdomain}
\end{figure}
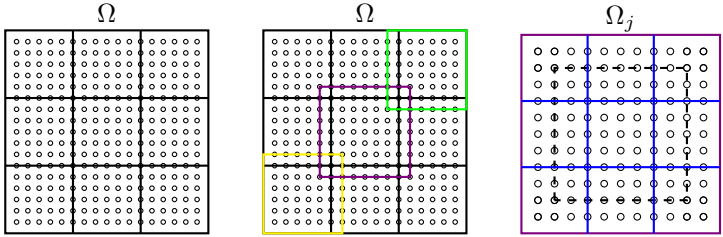

Figure \ref{Fig:S2S} compares several versions of the S2S method to the SHEM method. Specifically, we consider a S2S method with a coarse space made by eigenfunctions of $G$ (S2S-$G$), a S2S method with a coarse space obtained with the PCA procedure (S2S-PCA), and a S2S method with a coarse space which is inspired by the SHEM coarse space (S2S-HEM, that is S2S Harmonically Enriched Multiscale) and a S2S method with a coarse space obtained with the deep neural network approach (S2S-DNN).

In more detail, we create the HEM coarse space by computing harmonic functions and solving interface eigenvalue problems 
on each one-dimensional segment that forms the local discrete substructure. Let us recall that the SHEM coarse space is based on harmonic and spectral functions which are computed along the boundaries of a nonoverlapping decomposition, see \cite{gander2015analysis}. Then, the SHEM method extends these interface functions into the interior of the nonoverlapping subdomains as the method is naturally defined in volume. We do not need to perform this extra step of extending the functions in the neighboring subdomains. We report that we have also tried to build the a coarse space by simply restricting the volume functions of the SHEM coarse space onto the substructures and we observed a similar behavior compared to the HEM coarse space.
For the PCA approach, we generated $q=2\times$dim$V_c$  random vectors and we set $r=2$. The result we plot is averaged over 30 different random coarse spaces.
For the deep neural network, we used $k=4$ and $N=N^s$.

The size of the coarse space is set by the SHEM coarse space. In the top-left panel, we consider only multiscale functions without solving any eigenvalue problem along the interfaces. In the top-right panel, we include the first eigenfunctions on each interface, and on the bottom-central panel we include the first and the second eigenfunctions.
\begin{figure}
\centering
\includegraphics[scale=0.4]{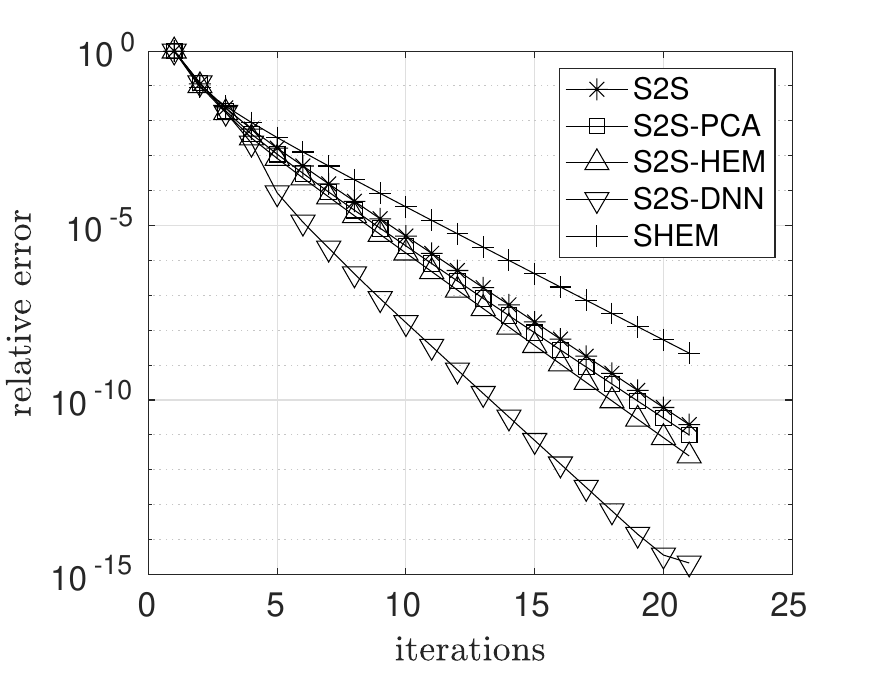}
\includegraphics[scale=0.4]{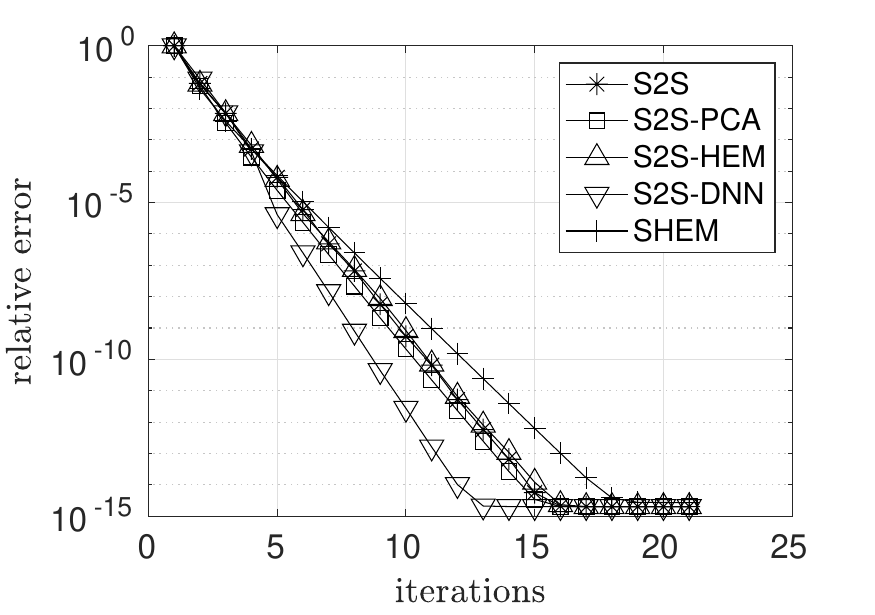}
\includegraphics[scale=0.4]{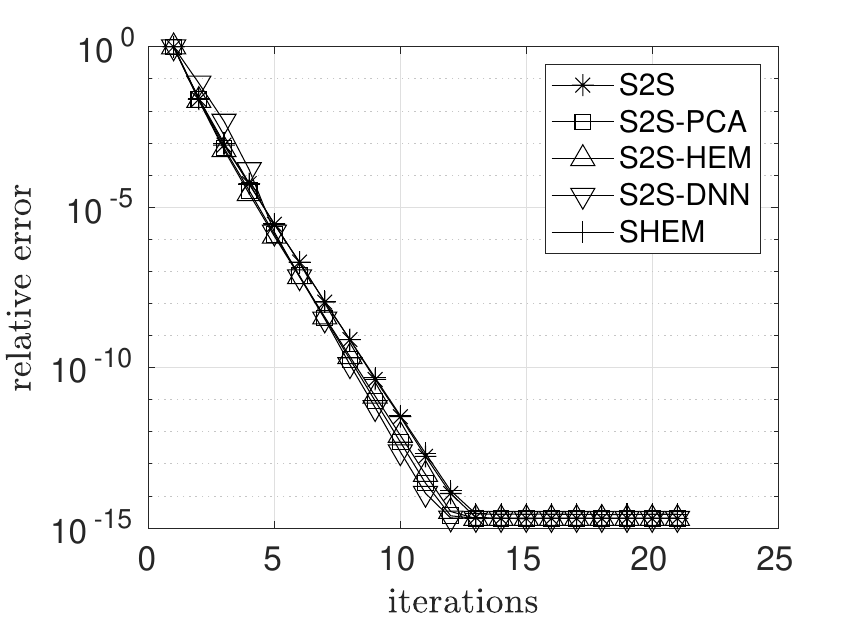}
\caption{Convergence behavior of the different methods for a Laplace equation with $N=16$, $\ell=4$ and $N_{ov}=2$. The dimension of the coarse space is 36 (top-left), 84 (top-right), 132 (bottom).}\label{Fig:S2S}
\end{figure}
In all cases we observe that the methods have a similar convergence, which is slightly faster for the substructured methods for smaller coarse spaces.
As we already remarked, S2S-$G$ is not necessarily the fastest.

\subsection{Diffusion problem with jumping diffusion coefficients}\label{sec:num:jump}	
In this paragraph, we test the S2S method for the solution of a diffusion equation 
$-\text{div}(\alpha \nabla u) = f$
in a square domain $\Omega:=(0,1)^2$ with $f:=\sin(4\pi x)\sin(2\pi y)\sin(2\pi xy)$. The domain $\Omega$ is decomposed into 16 non-overlapping subdomains and we suppose $\alpha=1$ everywhere except in some channels where $\alpha$ takes the values large values. Each non-overlapping subdomain is discretized with $N_{\text{sub}}=2^{2\ell}$ cells and enlarged by
$N_{ov}$ cells to create an overlapping decomposition with overlap $\delta=2N_{ov} h$. We use a finite-volume scheme and we assume that the jumps of the diffusion coefficients are aligned with the cell edges.
We consider two configurations represented in Figure \ref{fig:jumping}. 
\begin{figure}[]
\centering
\includegraphics[scale=0.28]{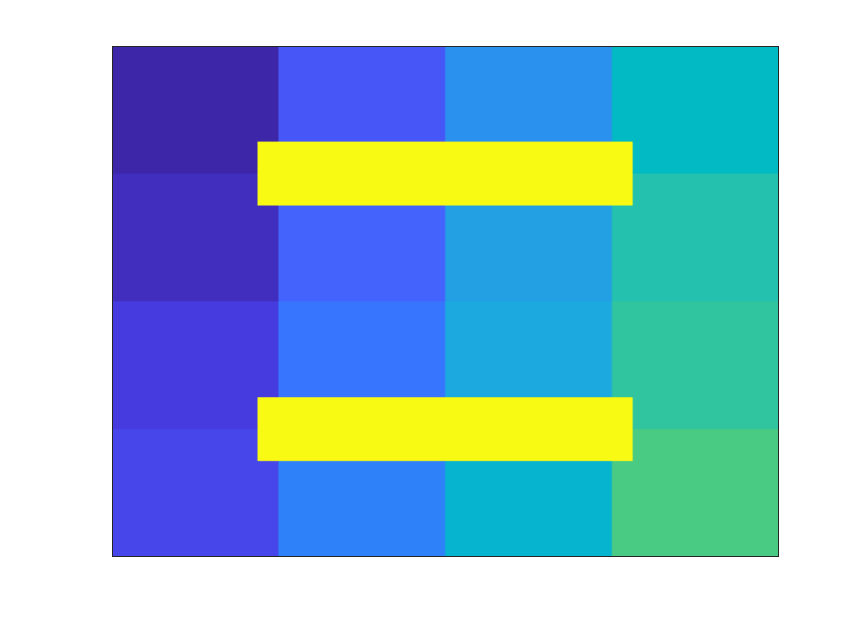}
\includegraphics[scale=0.28]{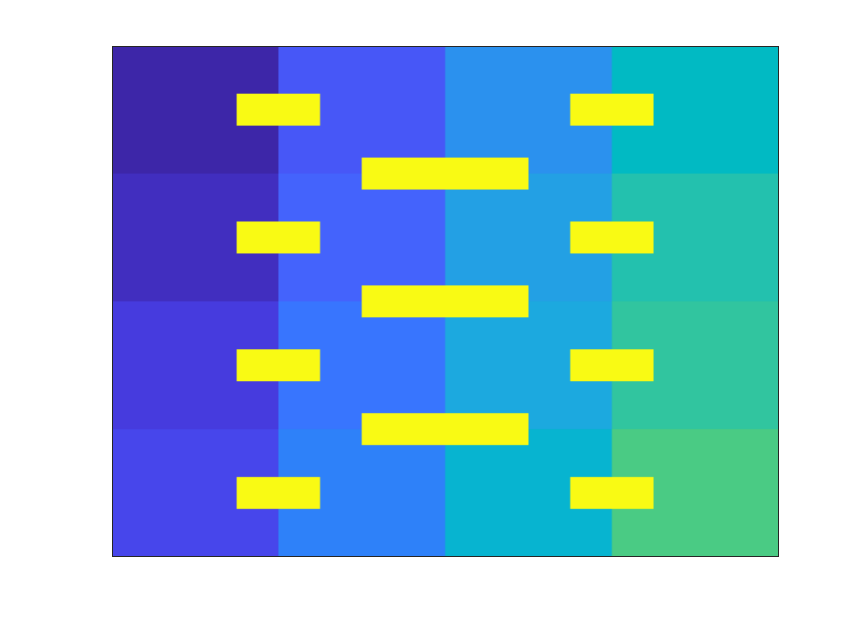}
\caption{Decomposition of $\Omega$ into 16 subdomains with two different patterns of channels.}
\label{fig:jumping}
\end{figure}

We first aim to validate numerically Theorem \ref{Ciaramella_mini_10_thm:perturb}. We consider the two channels configuration with $\alpha=5\cdot 10^3$, $\ell=4$, $N=4$. The first five eigenvalues are $\lambda_1=0.999$, $\lambda_2=-0.9989$, $\lambda_3=-0.99863$,$\lambda_4=0.99861$ and $\lambda_5=0.2392$.
We consider the coarse space $V_c:=\left\{ \psib_1+\varepsilon \psib_2\right\}$, where $\gamma=\langle \psib_1,\psib_2\rangle\approx 10^{-15}$ so that the two eigenvectors are orthogonal. As $\lambda_1$ and $\lambda_2$ have opposite signs, point (C) of Theorem \ref{Ciaramella_mini_10_thm:perturb} guarantees the existence 
of an $\widetilde{\varepsilon}$ such that $\rho(T(\widetilde{\varepsilon}))=|\lambda_3|<|\lambda_2|$.
Figure \ref{Fig:perturbedcoarsespace} confirms on the left panel that $|\lambda(\varepsilon,0)|$ reaches a zero for two values of $\varepsilon$. The right panel clearly shows that for several values of $\varepsilon$, $\rho(T(\varepsilon))=|\lambda_3|$. It is interesting to remark that, in this setting, choosing $\varepsilon=0$ (that is, a standard spectral coarse space) is actually the worse choice, as for any $\varepsilon\neq 0$ $\rho(T(\varepsilon))\leq |\lambda_2|$ as the proof of point (C) of Theorem \ref{Ciaramella_mini_10_thm:perturb} shows.

\begin{figure}[]
\centering
\includegraphics[scale=0.4]{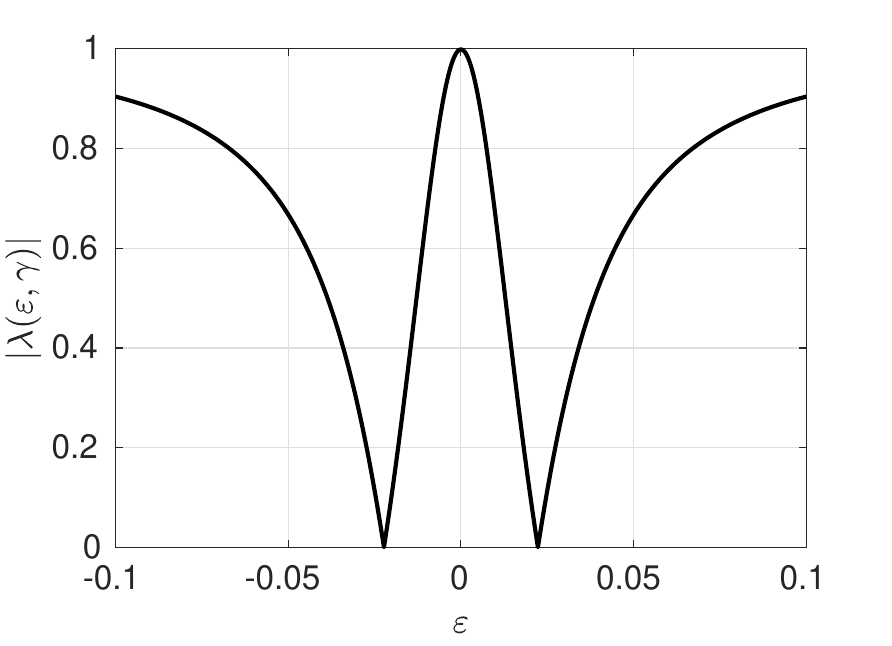}
\includegraphics[scale=0.4]{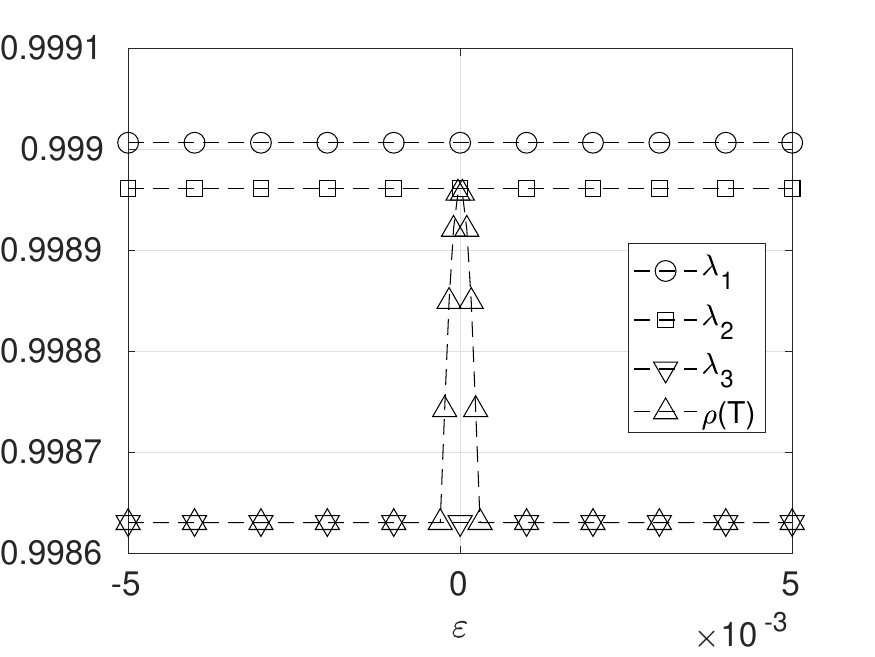}
\caption{Plot of the function $\varepsilon\mapsto |\lambda(\varepsilon,0)|$ on the left panel, comparison between $\rho(T(\varepsilon))$ and the first eigenvalues on the right panel.}
\label{Fig:perturbedcoarsespace}
\end{figure}

Second, we build the coarse space $V_c:=\left\{\psib_1+\varepsilon \psib_5\right\}$, where $\psib_i$ $i=1,5$ are the first and fifth eigenfunctions of $G$ and $\varepsilon=0.01$. Figure \ref{Fig:diverging} shows that the S2S method method with this specific choice of coarse space is diverging. Computing the spectral radius, we obtain $\rho(T)= 1.2322$. In this setting, we have $\gamma=-0.5628$  which replaced into the expression of $\lambda(\gamma,\varepsilon)$, together with the values of $\lambda_1$ and $\lambda_5$, leads to $\lambda(\gamma,\varepsilon)=1.2322=\rho(T)$. Indeed $\lambda(\gamma,\varepsilon)$ has a vertical asymptote in $\gamma^*=-0.1404$ as shown on the right panel of Figure \ref{Fig:diverging}. We can restore the convergence of the S2S method by a sufficient decrease of $\varepsilon$, that is by reducing the perturbation in the coarse space.
In a numerical implementation, this is obtained by performing $r\geq 1$ iterations of the smoother $G$ on the coarse space (see Corollary \ref{cor:corr_Vc}). Indeed it holds that
\[ G^r V_c=G^r\text{span}\left\{\psib_1+\varepsilon \psib_5\right\}=\text{span}\left\{\psib_1 +\frac{\lambda_5^r}{\lambda_1^r}\varepsilon \psib_5\right\}.\]
Applying twice the smoother in the case at hand, we get a new ``smoothed'' coarse space where the perturbation has size $\varepsilon^*=\frac{\lambda_5^2}{\lambda_1^2}\varepsilon=5.73\cdot 10^{-4}$ so that now $\lambda(\gamma,\varepsilon^*)=-0.0080$. We remark that $\lambda(\gamma,\varepsilon^*)$ is the convergence factor of $T$ on $\text{span}\left\{\psib_1+\varepsilon^* \psib_5\right\}$, so that the convergence of the S2S method is now determined by second largest eigenvalue of $T$, i.e. $\lambda_2=-0.9990$ as Figure \ref{Fig:diverging} shows.

\begin{figure}[]
\centering
\includegraphics[scale=0.4]{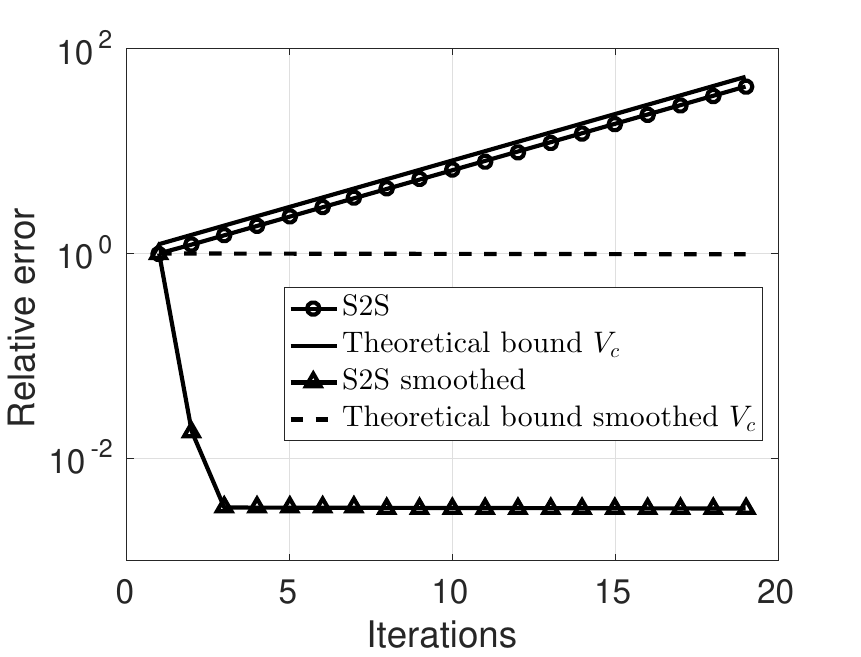}
\includegraphics[scale=0.4]{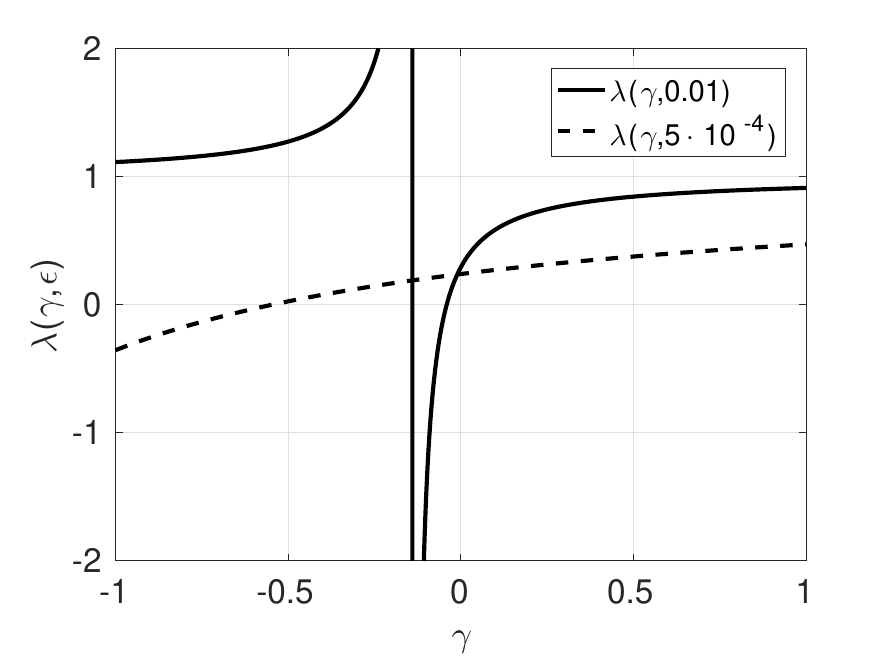}
\caption{Convergence curves for a S2S method with the coarse space $V_c:=\left\{\psib_1+\varepsilon \psib_5\right\}$ (left) and plot of the function $\lambda(\gamma,0.01)$ (right).}
\label{Fig:diverging}
\end{figure}

We then investigate the performances of the S2S methods and we compared them with the SHEM coarse space in the multiple channel configuration. We set $\ell=4$, $N=16$, which correspond to $N^v=4096$ degrees of freedom, and $N_{ov}=2$. Table \ref{tab:S2S} shows the number of iterations to reach a relative error smaller than $10^{-8}$ for the S2S-$G$, S2S-PCA, S2S-HEM and SHEM methods.
The relative error is computed with respect to the first iterate.
We consider coarse spaces of dimensions 84, 132 and 180, which, for the SHEM and S2S-HEM methods, correspond to multiscale coarse spaces enriched by respectively the first, second and third eigenvectors of the interface eigenvalues problems. For the PCA coarse space, we set $q=2N_c$ and $r=6$ if $\alpha=10^6$, $r=4$ if $\alpha=10^4$ and $r=2$ if $\alpha=10^2$. We remark that for smaller values of $r$, the S2S-PCA method diverges. This increase in the value of $r$ can be explained noticing that for the multichannel configuration, the smoother $G$ has several eigenvalues approximately 1 for large values of $\alpha$. Thus the PCA procedure, which essentially relies on a power method idea to approximate the image of $G$, suffers due to the presence of several clustered eigenvalues, and hence does not provide accurate approximations of the eigenfunctions of $G$.
Similarly also the HEM coarse space obtained by solving on each segment of a skeleton an eigenvalue problem could lead to a divergent method. Thus, to improve this coarse space, we apply few iterations of the smoother to obtain a better $V_c$.
Tables \ref{tab:S2S} and \ref{tab:S2S_krylov} report the number of iterations to reach a tolerance of $10^{-8}$ when the algorithms are used either as stationary methods or as preconditioners. We remark that all spectral methods have very similar performance, and all methods are robust with respect to the strength of the jumps.

\begin{table}[]
\centering
\begin{small}
\setlength{\tabcolsep}{5pt}
\begin{tabular}{| c | c | c | c | c|}
\hline
$\alpha$ &  S2S-$G$ & S2S-PCA &  S2S-EHM & SHEM  \\ \hline
$10^2$ & 11-9-7  &  14-8-7 &  15-8-7& 16-10-8  \\
$10^4$ &  11-9-7 &  13-9-7 & 10-9-7 & 16-10-8 \\
$10^6$ & 12-10-8 &  18-9-7 &  12-9-8 & 16-10-8\\ \hline
\end{tabular}
\begin{tabular}{| c | c | c | c | c|}
\hline
$\alpha$ &  S2S-$G$ & S2S-PCA &  S2S-EHM & SHEM  \\ \hline
$10^2$ & 10-9-7  &  11-8-7 &  15-9-7& 12-12-8  \\
$10^4$ &  10-9-7 &  11-9-7 & 11-10-7 & 12-12-8 \\
$10^6$ & 10-9-7 &  14-8-7 &  11-10-7 & 13-10-7\\ \hline
\end{tabular}
\end{small}
\caption{For each spectral method and value of $\alpha$, we report the number of iterations to reach a relative error smaller than $10^{-8}$ with a coarse space of dimension 84 (left), 132 (center) and 180 (right).
The discretization parameters are $N_v=4096 $ and $N_{ov}=2$. The left table refers to a two-channels configuration, the right table refers to the multiple channel configuration depicted in Figure \ref{fig:jumping}.}
\label{tab:S2S}
\end{table}

\begin{table}[]
\centering
\begin{small}
\setlength{\tabcolsep}{5pt}
\begin{tabular}{| c | c | c | c | c|}
\hline
$\alpha$ &  S2S-$G$ & S2S-PCA &  S2S-EHM & SHEM  \\ \hline
$10^2$ & 7-5-4  &  8-5-4 &  7-5-4 & 7-6-4  \\
$10^4$ &  5-5-4 &  6-5-5 & 5-5-4 & 7-6-4 \\
$10^6$ & 5-5-3 &  5-4-3 &  5-4-4 & 7-6-4\\ \hline
\end{tabular}
\begin{tabular}{| c | c | c | c | c|}
\hline
$\alpha$ &  S2S-$G$ & S2S-PCA &  S2S-EHM & SHEM  \\ \hline
$10^2$ & 7-5-4  &  8-6-5 &  8-6-4 & 7-6-4  \\
$10^4$ &  7-7-4 &  8-6-5 & 7-7-5 & 7-5-5 \\
$10^6$ & 8-6-6 &  9-6-6 &  7-7-6 & 7-5-5\\ \hline
\end{tabular}
\end{small}
\caption{Number of iterations performed by GMRES preconditioned by different methods and for several value of $\alpha$ to reach a relative error smaller than $10^{-8}$ with a coarse space of dimension 84 (left), 132 (center) and 180 (right).
The discretization parameters are $N_v=4096 $ and $N_{ov}=2$. The left table refers to a two-channels configuration, the right table refers to the multiple channel configuration depicted in Figure \ref{fig:jumping}.}
\label{tab:S2S_krylov}
\end{table}

\section{Conclusions}\label{sec:conclusions}
In this work we introduced a new computational framework of two-level substructured DD methods.
This is called S2S and is based on coarse spaces defined exclusively on some interfaces provided by the decomposition of the domain. 
We presented a broader convergence analysis for two-level iterative methods, which covers the proposed substructured framework as a special case.
The analysis pushes forward the current understanding of asymptotic optimality of coarse spaces.
From the computational point of view, we have discussed approaches based on the PCA and deep neural networks for the numerical computation of efficient coarse spaces.
Finally, the effectiveness of our new methods is confirmed by extensive numerical experiments, where
stationary elliptic problems (with possibly highly jumping diffusion coefficients) are efficiently solved.

\section{Appendix}\label{sec:appendix}

In this Appendix, important implementation details
of our substructured two-level methods are discussed. We reformulate Algorithm \ref{two-level} in equivalent forms that are computationally more efficient.
This is essential to make our methods computationally equal or more efficient than other
existing strategies.

As already remarked in Section \ref{sec:num_exp}, a naive implementation of Algorithm \ref{two-level} 
would lead to a quite expensive method as the computation of the residual involves a matrix multiplication with $A$, 
which requires to perform subdomain solves. Hence, one would need two subdomain solves per iteration.
To avoid this extra cost, we use the special form of the matrix $A=I-G$
and propose two new versions of Algorithm \ref{two-level}. 
These are called S2S-B1 and S2S-B2 and
given by Algorithm \ref{two-level-B1} and Algorithm \ref{two-level-B2}.
\begin{figure}[t]
\begin{minipage}{0.45\textwidth}
\def\NoNumber#1{{\def\alglinenumber##1{}\State #1}\addtocounter{ALG@line}{-1}}
\begin{algorithm}[H]
\setlength{\columnwidth}{\linewidth}
\caption{S2S-B1}
\vspace*{-1mm}
\begin{algorithmic}[1]\label{two-level-B1}
\REQUIRE ${\bf u}^{0}$. \\
\STATE ${\bf u}^{1}=G{\bf u}^{0} +{\bf b}$,  \\
\STATE ${\bf v}=G{\bf u}^{1}$, \\
\STATE ${\bf r} = {\bf b}-{\bf u}^{1}+{\bf v}$, \\
\STATE ${\bf{d}}=A_c^{-1}R{\bf r}$,  \\
\STATE ${\bf u}^{0} = {\bf u}^1 + P{\bf{d}}$,  \\
{\bf Iterations}:
\STATE ${\bf u}^{1}={\bf v} +\widetilde{P}{\bf d} +{\bf b}$,\\ 
\STATE ${\bf v}=G{\bf u}^{1}$,\\
\STATE ${\bf r} = {\bf b}-{\bf u}^{1}+{\bf v}$, \\
\STATE ${\bf{d}}=A_c^{-1}R{\bf r}$,  \\
\STATE ${\bf u}^{0} = {\bf u}^{1} + P{\bf{d}}$,\\
\STATE Repeat from 6 to 10 until convergence.
\end{algorithmic}
\end{algorithm}
\end{minipage}
\begin{minipage}{0.45\textwidth}
\vspace*{-24mm}
\begin{algorithm}[H]
\setlength{\columnwidth}{\linewidth}
\caption{S2S-B2}
\begin{algorithmic}[1]\label{two-level-B2}
\REQUIRE ${\bf u}^{0}$ and set $n=1$, \\
\STATE ${\bf v}=G{\bf u}^{0}$,  \\
\STATE ${\bf r} = {\bf b}-{\bf u}^{0}+{\bf v}$, \\
\STATE ${\bf{d}}=A_c^{-1}R{\bf r}$,  \\
\STATE ${\bf u}^{1} = {\bf v} + \widetilde{P}{\bf{d}} + {\bf b}$ ,  \\
\STATE Set ${\bf u}^{0}={\bf u}^{1}$ and repeat from 1 to 5 until convergence.
\end{algorithmic}
\end{algorithm}
\end{minipage}
\end{figure}
The relations between S2S, S2S-B1 and S2S-B2 are given in the following theorem.

\begin{theorem}[Equivalence between S2S, S2S-B2 and S2S-B1]\label{Th:equivalence_S2S-B1-B2}
\text{ }\\
\begin{itemize}
\item[{\rm (a)}] Algorithm \ref{two-level-B1} generates the same iterates of Algorithm \ref{two-level}.
\item[{\rm (b)}] Algorithm \ref{two-level-B2} corresponds to the stationary iterative method ${\bf u}^n=G(\mathbb{I}-PA_c^{-1}RA){\bf u}^{n-1} + \widetilde{M}{\bf b},$
where $G(\mathbb{I}-PA_c^{-1}RA)$ is the iteration matrix and $\widetilde{M}$ the relative preconditioner.
Moreover, Algorithm \ref{two-level-B2} and Algorithm \ref{two-level-B1} have the same convergence behavior.
\end{itemize}
\end{theorem}
\begin{proof} 
 For simplicity, we suppose to work with the error equation and thus ${\bf b}=0$.
We call $\widetilde{{\bf u}}^0$ the output of the first five steps of Algorithm \ref{two-level-B1} and with $\widehat{{\bf u}}^0$ the output of Algorithm \ref{two-level}. Then given a initial guess ${\bf u}^0$, we have
\begin{equation*}
\widetilde{{\bf u}}^0 = {\bf u}^1 + P{\bf d} 
= {\bf u}^1 + PA_c^{-1}R(-{\bf u}^1 + {\bf v})
= G{\bf u}^0 + PA_c^{-1}R(-AG{\bf u}^0)
=(\mathbb{I}-PA_c^{-1}RA)G{\bf u}^0=\widehat{{\bf u}}^0.
\end{equation*}
Similar calculations show that also steps 6-10 of S2S-B1 are equivalent to an iteration of \ref{two-level}.
For the second part of the Theorem, we write the iteration matrix for Algorithm \ref{two-level-B2} as
\[ {\bf u}^{1}= {\bf v} +\widetilde{P}{\bf d}=G {\bf u}^0 + GPA_c^{-1}R(-A{\bf u}^0)=G(\mathbb{I}-PA_c^{-1}RA){\bf u}^0.\]
Hence, Algorithm \ref{two-level-B2} performs a post-smoothing step instead of a pre-smoothing step as 
Algorithm \ref{two-level-B1} does. The method still has the same convergence behavior since the matrices 
$G(\mathbb{I}-PA_c^{-1}RA)$ and $(\mathbb{I}-PA_c^{-1}RA)G$ have the same
eigenvalues\footnote{Given two matrices $A$ and $B$, $AB$ and $BA$ share the same non-zero eigenvalues.}.
\end{proof}

Notice that Algorithm \ref{two-level-B1} requires for the first iteration
two applications of the smoothing operator $G$, namely two subdomains solves.
The next iterations, given by Steps 6-10, need only one application
of the smoothing operator $G$. Theorem \ref{Th:equivalence_S2S-B1-B2} {\rm (a)}
shows that Algorithm \ref{two-level-B1} is equivalent to Algorithm \ref{two-level}.
This means that each iteration after the first one of Algorithm \ref{two-level-B1}
is computationally less expensive than one iteration of a volume two-level DD method.
Since two-level DD methods perform generally few iteration, it could be important to
get rid of the expensive first iteration. For this reason, we introduce
Algorithm \ref{two-level-B2}, which overcome the problem of the first iteration.
Theorem \ref{Th:equivalence_S2S-B1-B2} {\rm (b)} guarantees that Algorithm \ref{two-level-B2}
is exactly an S2S method with no pre-smoothing and one post-smoothing step.
Moreover, it has the same convergence behavior of Algorithm \ref{two-level-B1}.

We wish to remark that, the reformulations S2S-B1 and S2S-B2 require to store the
matrix $\widetilde{P}:= GP$, which is anyway needed in the assembly phase of the coarse matrix,
hence no extra cost is required, if compared to a volume two-level DD method.
Finally, we stress that these implementation tricks can be readily generalized to a general number
of pre- and post-smoothing steps.

\bibliographystyle{siamplain}
\bibliography{references}

\begin{thebibliography}{10}

\bibitem{Aarnes2002}
{\sc J.~Aarnes and T.~Y. Hou}, {\em Multiscale domain decomposition methods for
  elliptic problems with high aspect ratios}, Acta Math. Appl. Sin., 18 (2002),
  pp.~63--76.

\bibitem{Bjorstad2018}
{\sc P.~Bjorstad, M.~J. Gander, A.~Loneland, and T.~Rahman}, {\em {Does SHEM
  for Additive Schwarz work better than predicted by its condition number
  estimate?}}, Domain Decomposition Methods in Science and Engineering XXIV,
  LNCSE, Springer, -- (2018), pp.~129--138.

\bibitem{Brezina2005}
{\sc M.~Brezina, R.~Falgout, S.~MacLachlan, T.~Manteuffel, S.~McCormick, and
  J.~Ruge}, {\em Adaptive smoothed aggregation {($\alpha$ SA)} multigrid}, SIAM
  Rev., 47 (2005), pp.~317--346.

\bibitem{Chaouqui2018}
{\sc F.~Chaouqui, G.~Ciaramella, M.~J. Gander, and T.~Vanzan}, {\em On the
  scalability of classical one-level domain-decomposition methods}, Vietnam J.
  Math., 46 (2018), pp.~1053--1088.

\bibitem{Chaouqui2018Coarse}
{\sc F.~Chaouqui, M.~J. Gander, and K.~Repiquet}, {\em A coarse space to remove
  the logarithmic dependency in {N}eumann-{N}eumann methods}, Domain
  Decomposition Methods in Science and Engineering XXIV, LNCSE, Springer, --
  (2018), pp.~159--168.

\bibitem{Chaouqui2019}
{\sc F.~Chaouqui, M.~J. Gander, and K.~Santugini-Repiquet}, {\em A local coarse
  space correction leading to a well-posed continuous neumann-neumann method in
  the presence of cross points}, in Domain Decomposition Methods in Science and
  Engineering XXV, Cham, 2020, Springer, pp.~83--91.

\bibitem{Gabriele-Martin}
{\sc G.~Ciaramella and M.~J. Gander}, {\em Iterative Methods and
  Preconditioners for Systems of Linear Equations}, in preparation.

\bibitem{CiaramellaGander}
{\sc G.~Ciaramella and M.~J. Gander}, {\em Analysis of the parallel {S}chwarz
  method for growing chains of fixed-sized subdomains: Part {I}}, SIAM J.
  Numer. Anal., 55 (2017), pp.~1330--1356.

\bibitem{CiaramellaGander2}
{\sc G.~Ciaramella and M.~J. Gander}, {\em {Analysis of the parallel Schwarz
  method for growing chains of fixed-sized subdomains: Part II}}, {SIAM J.
  Numer. Anal.}, 56 (3) (2018), pp.~1498--1524.

\bibitem{CiaramellaGander3}
{\sc G.~Ciaramella and M.~J. Gander}, {\em {Analysis of the parallel Schwarz
  method for growing chains of fixed-sized subdomains: Part III}}, {Electron.
  Trans. Numer. Anal.}, 49 (2018), pp.~201--243.

\bibitem{DDLOGO}
{\sc G.~Ciaramella and M.~J. Gander}, {\em Happy 25th anniversary {DDM}! ...
  {But} how fast can the {Schwarz} method solve your logo?}, in Domain
  Decomposition Methods in Science and Engineering XXV, Cham, 2020, Springer,
  pp.~92--99.

\bibitem{CiaramellaRefl}
{\sc G.~Ciaramella, M.~J. Gander, L.~Halpern, and J.~Salomon}, {\em Methods of
  reflections: relations with schwarz methods and classical stationary
  iterations, scalability and preconditioning.}, SMAI J. Comput. Appl. Math., 5
  (2019), pp.~161--193.

\bibitem{CiaramellaGanderMamooler}
{\sc G.~Ciaramella, M.~J. Gander, and P.~Mamooler}, {\em The domain
  decomposition method of {Bank} and {Jimack} as an optimized {Schwarz}
  method}, in Domain Decomposition Methods in Science and Engineering XXV,
  Cham, 2020, Springer, pp.~285--293.

\bibitem{CHS2}
{\sc G.~Ciaramella, M.~Hassan, and B.~Stamm}, {\em On the scalability of the
  {S}chwarz method}, SMAI J. Comput. Appl. Math., 6 (2019).

\bibitem{CHS1}
{\sc G.~Ciaramella, M.~Hassan, and B.~Stamm}, {\em On the scalability of the
  parallel {Schwarz} method in one-dimension}, in Domain Decomposition Methods
  in Science and Engineering XXV, Cham, 2020, Springer, pp.~151--158.

\bibitem{CiaramellaHoefer}
{\sc G.~Ciaramella and R.~M. H{\"o}fer}, {\em Non-geometric convergence of the
  classical alternating {Schwarz} method}, in Domain Decomposition Methods in
  Science and Engineering XXV, Cham, 2020, Springer, pp.~193--201.

\bibitem{Dohrmann2008}
{\sc C.~R. Dohrmann, A.~Klawonn, and O.~B. Widlund}, {\em A family of energy
  minimizing coarse spaces for overlapping {S}chwarz preconditioners}, in
  Domain Decomposition Methods in Science and Engineering XVII, 2008,
  pp.~247--254.

\bibitem{DoleanBook}
{\sc V.~Dolean, P.~Jolivet, and F.~Nataf}, {\em An Introduction to Domain
  Decomposition Methods}, SIAM, Philadelphia, PA, 2015.

\bibitem{DoleanDTN}
{\sc V.~Dolean, F.~Nataf, R.~Scheichl, and N.~Spillane}, {\em Analysis of a
  two-level {S}chwarz method with coarse spaces based on local
  {D}irichlet-to-{N}eumann maps}, Comput. Meth. in Appl. Math., 12 (2012),
  pp.~391--414.

\bibitem{Dubois2012}
{\sc O.~Dubois, M.~J. Gander, S.~Loisel, A.~St-Cyr, and D.~B. Szyld}, {\em The
  optimized {S}chwarz method with a coarse grid correction}, SIAM J. Sci.
  Comput., 34 (2012), pp.~421--458.

\bibitem{Galvis3}
{\sc Y.~Efendiev, J.~Galvis, R.~Lazarov, and J.~Willems}, {\em Robust domain
  decomposition preconditioners for abstract symmetric positive definite
  bilinear forms}, ESAIM Math. Model. Numer. Anal., 46 (2012), pp.~1175--1199.

\bibitem{friedman1982foundations}
{\sc A.~Friedman}, {\em Foundations of Modern Analysis}, Dover Books on
  Mathematics Series, Dover, 1982.

\bibitem{Galvis}
{\sc J.~Galvis and Y.~Efendiev}, {\em Domain decomposition preconditioners for
  multiscale flows in high-contrast media}, Multiscale Model. Sim., 8 (2010),
  pp.~1461--1483.

\bibitem{Galvis2}
{\sc J.~Galvis and Y.~Efendiev}, {\em Domain decomposition preconditioners for
  multiscale flows in high contrast media: Reduced dimension coarse spaces},
  Multiscale Model. Sim., 8 (2010), pp.~1621--1644.

\bibitem{Gander1}
{\sc M.~J. Gander}, {\em Optimized {S}chwarz methods}, SIAM J. Numer. Anal., 44
  (2006), pp.~699--731.

\bibitem{Gander:2008}
{\sc M.~J. Gander}, {\em {Schwarz methods over the course of time}}, Electron.
  Trans. Numer. Anal., 31 (2008), pp.~228--255.

\bibitem{Gander2011}
{\sc M.~J. Gander}, {\em On the influence of geometry on optimized {S}chwarz
  methods}, SeMA Journal, 53 (2011), pp.~71--78.

\bibitem{gander2014new}
{\sc M.~J. Gander, L.~Halpern, and K.~Repiquet}, {\em A new coarse grid
  correction for {RAS/AS}}, in Domain Decomposition Methods in Science and
  Engineering XXI, Springer, 2014, pp.~275--283.

\bibitem{GHS2018}
{\sc M.~J. Gander, L.~Halpern, and K.~Repiquet}, {\em On optimal coarse spaces
  for domain decomposition and their approximation}, accepted for Domain
  Decomposition Methods in Science and Engineering XXIV, LNCSE, Springer,
  (2018), pp.~271--280.

\bibitem{gander2017shem}
{\sc M.~J. Gander and A.~Loneland}, {\em {SHEM}: An optimal coarse space for
  {RAS} and its multiscale approximation}, in Domain Decomposition Methods in
  Science and Engineering XXIII, Springer, 2017, pp.~313--321.

\bibitem{gander2015analysis}
{\sc M.~J. Gander, A.~Loneland, and T.~Rahman}, {\em Analysis of a new
  harmonically enriched multiscale coarse space for domain decomposition
  methods}, preprint arXiv:1512.05285,  (2015).

\bibitem{gander2019song}
{\sc M.~J. Gander and B.~Song}, {\em Complete, optimal and optimized coarse
  spaces for additive {S}chwarz}, in Domain Decomposition Methods in Science
  and Engineering XXIV, Springer, 2018.

\bibitem{GanderVanC2019}
{\sc M.~J. Gander and S.~Van~Criekingen}, {\em New coarse corrections for
  optimized restricted additive {Schwarz} using petsc}, in Domain Decomposition
  Methods in Science and Engineering XXV, Cham, 2020, Springer, pp.~483--490.

\bibitem{GanderVanzan19}
{\sc M.~J. Gander and T.~Vanzan}, {\em Multilevel optimized {Schwarz} methods},
  SIAM J. Sci. Comp., 42 (2020), pp.~A3180--A3209.

\bibitem{GanderXu1}
{\sc M.~J. Gander and Y.~Xu}, {\em Optimized {S}chwarz methods for circular
  domain decompositions with overlap}, SIAM J. Numer. Anal., 52 (2014),
  pp.~1981--2004.

\bibitem{govloan2013}
{\sc G.~H. Golub and C.~F. {Van Loan}}, {\em Matrix Computations (Fourth
  Edition)}, Johns Hopkins Studies in the Mathematical Sciences, Johns Hopkins
  University Press, Baltimore, MD, 2013.

\bibitem{Graham2007}
{\sc I.~G. Graham, P.~O. Lechner, and R.~Scheichl}, {\em Domain decomposition
  for multiscale {PDEs}}, Numer. Math., 106 (2007), pp.~589--626.

\bibitem{greenbaum1997iterative}
{\sc A.~Greenbaum}, {\em Iterative Methods for Solving Linear Systems},
  Frontiers in Applied Mathematics, SIAM, Philadelphia PA, 1997.

\bibitem{Gubisch2017}
{\sc M.~Gubisch and S.~Volkwein}, {\em Chapter 1: Proper Orthogonal
  Decomposition for Linear-Quadratic Optimal Control}, SIAM, Computational
  Science \& Engineering, Philadelphia, PA, 2017, pp.~3--63.

\bibitem{hackbusch2013multi}
{\sc W.~Hackbusch}, {\em Multi-Grid Methods and Applications}, Series in
  Computational Mathematics, Springer Berlin Heidelberg, 2013.

\bibitem{Hackbusch_book}
{\sc W.~Hackbusch}, {\em {Iterative solution of large sparse systems of
  equations}}, vol.~95 of Applied mathematical sciences, Springer, Cham, second
  edition~ed., 2016.

\bibitem{Klawonn}
{\sc A.~Heinlein, A.~Klawonn, J.~Knepper, and O.~Rheinbach}, {\em Multiscale
  coarse spaces for overlapping {S}chwarz methods based on the {ACMS} space in
  {2D}}, Electron. Trans. Numer. Anal., 48 (2018), pp.~156--182.

\bibitem{hutchinson1989stochastic}
{\sc M.~F. Hutchinson}, {\em {A stochastic estimator of the trace of the
  influence matrix for Laplacian smoothing splines}}, Commun. Stat.-Simul. C.,
  18 (1989), pp.~1059--1076.

\bibitem{katrutsa2017deep}
{\sc A.~Katrutsa, T.~Daulbaev, and I.~Oseledets}, {\em Deep multigrid: learning
  prolongation and restriction matrices}, arXiv preprint arXiv:1711.03825,
  (2017).

\bibitem{Klawonn2015}
{\sc A.~Klawonn, P.~Radtke, and O.~Rheinbach}, {\em {FETI-DP} methods with an
  adaptive coarse space}, SIAM J. Numer. Anal., 53 (2015), pp.~297--320.

\bibitem{lax2002functional}
{\sc P.~Lax}, {\em Functional Analysis}, Pure and Applied Mathematics: A Wiley
  Series of Texts, Monographs and Tracts, Wiley, 2002.

\bibitem{lions1972non}
{\sc J.~Lions and E.~Magenes}, {\em Non-homogeneous Boundary Value Problems and
  Applications (Vol I)}, Die Grundlehren der mathematischen Wissenschaften,
  Springer-Verlag Berlin Heidelberg, 1972.

\bibitem{Lions1}
{\sc P.~L. Lions}, {\em On the {S}chwarz alternating method. {I}}, First
  international symposium on domain decomposition methods for partial
  differential equations,  (1988), pp.~1--42.

\bibitem{Lions2}
{\sc P.~L. Lions}, {\em {On the Schwarz alternating method. II. Stochastic
  interpretation and other properties}}, in Second International Symposium on
  Domain Decomposition Methods for Partial Differential Equations, 1989,
  pp.~47--70.

\bibitem{quarteroni1999domain}
{\sc A.~Quarteroni and A.~Valli}, {\em Domain Decomposition Methods for Partial
  Differential Equations}, Numerical Mathematics and Scientific Computation,
  Oxford Science Publications, 1999.

\bibitem{Spillane2011}
{\sc N.~Spillane, V.~Dolean, P.~Hauret, F.~Nataf, C.~Pechstein, and
  R.~Scheichl}, {\em A robust two-level domain decomposition preconditioner for
  systems of {PDEs}}, C. R. Math., 349 (2011), pp.~1255 -- 1259.

\bibitem{Spillane2014}
{\sc N.~Spillane, V.~Dolean, P.~Hauret, F.~Nataf, C.~Pechstein, and
  R.~Scheichl}, {\em Abstract robust coarse spaces for systems of {PDEs} via
  generalized eigenproblems in the overlaps}, Numer. Math., 126 (2014),
  pp.~741--770.

\bibitem{tartar2007introduction}
{\sc L.~Tartar}, {\em An Introduction to Sobolev Spaces and Interpolation
  Spaces}, Lecture Notes of the Unione Matematica Italiana, Springer Berlin
  Heidelberg, 2007.

\bibitem{ToselliWidlund}
{\sc A.~Toselli and O.~Widlund}, {\em Domain Decomposition Methods: Algorithms
  and Theory}, vol.~34 of Series in Computational Mathematics, Springer, New
  York, 2005.

\bibitem{xu_zikatanov_2017}
{\sc J.~Xu and L.~Zikatanov}, {\em Algebraic multigrid methods}, Acta Numer.,
  26 (2017), pp.~591--721.

\bibitem{Zampini2017}
{\sc S.~Zampini and X.~Tu}, {\em Multilevel balancing domain decomposition by
  constraints deluxe algorithms with adaptive coarse spaces for flow in porous
  media}, SIAM J. Sci. Comput., 39 (2017), pp.~A1389--A1415.

\end{thebibliography}

\end{document}